\documentclass[final,reqno]{siamltex}
\usepackage{latexsym,amsmath,amssymb,amsfonts,mathrsfs}
\usepackage{epsf,graphicx,epsfig,color,cite,cases}
\usepackage{subfigure,graphics,multirow,marginnote,graphicx}
\newcommand{\R}{{\mathbb R}}

\newcommand{\Sp}{{\mathbb S}}

\newcommand{\no}{\nonumber}
\newcommand{\be}{\begin{eqnarray}}
\newcommand{\ben}{\begin{eqnarray*}}
\newcommand{\en}{\end{eqnarray}}
\newcommand{\enn}{\end{eqnarray*}}
\newcommand{\ba}{\backslash}
\newcommand{\pa}{\partial}

\newcommand{\G}{\Gamma}

\newcommand{\vep}{\varepsilon}
\newcommand{\Om}{\Omega}

\newcommand{\al}{\alpha}

\newcommand{\ol}{\overline}

\newtheorem{remark}[theorem]{Remark}
\newtheorem{algorithm}{Algorithm}[section]

\begin{document}
\renewcommand{\theequation}{\arabic{section}.\arabic{equation}}
\title{\bf A direct imaging method for inverse scattering by unbounded rough surfaces
}
\author{Xiaoli Liu\thanks{Academy of Mathematics and Systems Science, Chinese Academy of Sciences,
Beijing 100190, China and School of Mathematical Sciences, University of Chinese
Academy of Sciences, Beijing 100049, China ({\tt liuxiaoli@amss.ac.cn})}
\and Bo Zhang\thanks{LSEC, NCMIS and Academy of Mathematics and Systems Science, Chinese Academy of
Sciences, Beijing, 100190, China and School of Mathematical Sciences, University of Chinese
Academy of Sciences, Beijing 100049, China ({\tt b.zhang@amt.ac.cn})}
\and Haiwen Zhang\thanks{NCMIS and Academy of Mathematics and Systems Science, Chinese Academy of Sciences,
Beijing 100190, China ({\tt zhanghaiwen@amss.ac.cn})}}
\date{}

\maketitle

\begin{abstract}
This paper is concerned with the inverse scattering problem by an unbounded rough surface.
A direct imaging method is proposed to reconstruct the rough surface from the scattered near-field
Cauchy data generating by point sources and measured on a horizontal straight line segment at a finite
distance above the rough surface.
Theoretical analysis of the imaging algorithm is given for the case of a penetrable rough surface,
but the imaging algorithm also works for impenetrable surfaces with Dirichlet or impedance boundary conditions.
Numerical experiments are presented to show that the direct imaging algorithm
is fast, accurate and very robust with respect to noise in the data.
\end{abstract}

\begin{keywords}
Inverse scattering, unbounded rough surfaces, Cauchy data, Dirichlet boundary conditions,
impedance boundary conditions, transmission conditions
\end{keywords}

\begin{AMS}
78A46, 35P25
\end{AMS}

\pagestyle{myheadings}
\thispagestyle{plain}
\markboth{X. Liu, B. Zhang, and H. Zhang}{Direct imaging of unbounded rough surfaces}

\section{Introduction}\label{sec1}
\setcounter{equation}{0}

The ability to effectively find the geometrical information of unknown rough surfaces from the knowledge
of the scattered wave field is of great importance in various applications such as
radar and sonar detection, remote sensing, geophysics and nondestructive testing.

The aim of this paper is to study the inverse scattering problem by an unbounded rough surface, and in particular,
the imaging of the rough surface from the scattered near-field Cauchy data.
The wave propagation is governed by the Helmholtz equation. See Fig. \ref{fig0} for the problem geometry.
In this paper, we are restricted to the two-dimensional case for simplicity. However,
our imaging method and its analysis can be generalized to the three-dimensional case with appropriate modifications.

\begin{figure}[htbp]
\centering
\subfigure{\includegraphics[width=2.5in]{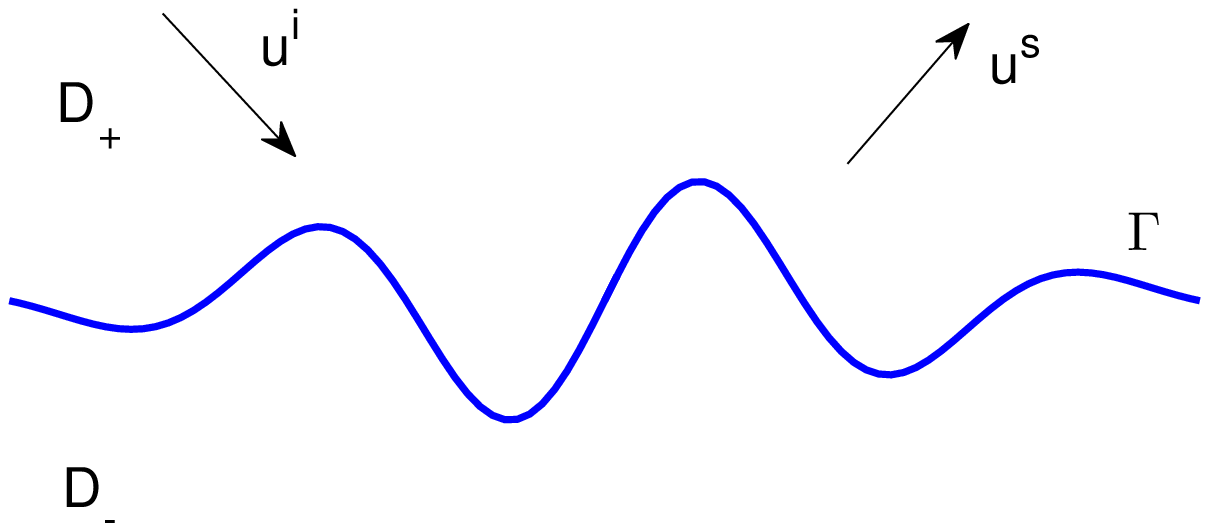}}
\subfigure{\includegraphics[width=2.5in]{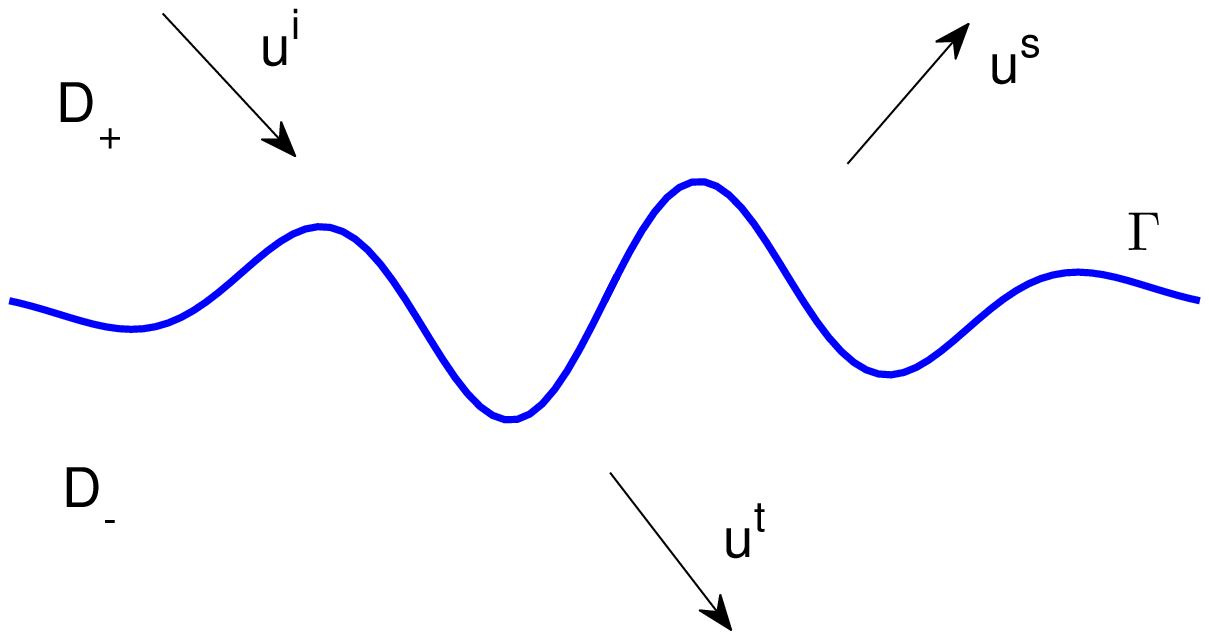}}
\caption{The scattering problems by an impenetrable rough surface (left) or
a penetrable rough surface (right).
}\label{fig0}
\end{figure}

Many optimization or iteration methods have been developed for inverse scattering problems \cite{BL05,LS14,LSZ16,vvA99}.
This kind of methods can recover the unknown surface very accurately, but they are time-consuming and also need
to know the physical property of the unknown surface in advance.
Recently, many methods have been proposed to avoid the huge computation. For example, the reverse time migration method,
the direct sampling method and the orthogonality sampling method are proposed
to reconstruct bounded obstacles in \cite{CCH13,IJZ12,P10}, respectively.
For the case of unbounded rough surfaces, a fast super-resolution method was proposed in \cite{BL13,BL14},
based on the transformed field expansion technique,
under the assumption that the rough surface is a small and smooth deformation of a plane surface.
Further, many related inverse scattering problems by unbounded surfaces have been extensively studied
for the case when the unbounded surface is a local perturbation of a plane surface \cite{BL11,DLLY,ZZ13},
for the case when the unbounded surface is periodic \cite{AK03,BL14IP}
and for the time-dependent case \cite{BP09,CL05}.

In this paper, we propose a direct imaging method to reconstruct the unbounded rough surface from the
scattered near-field Cauchy data generating by point sources and measured on a horizontal straight line
segment at a finite distance above the rough surface.
A main feature of our imaging method is its capability of depicting the profile of the surface
only through computing the inner products of the measured data and the fundamental solution
in the homogeneous background medium at each sampling point, leading to very cheap computation cost.
Further, our method does not require a prior knowledge of the physical property of the surface,
that is, the type of boundary conditions on the rough surface does not need to know in advance.
Thus, our imaging method works for both penetrable and impenetrable rough surfaces.
Furthermore, numerical experiments show that our imaging method can give an accurate and reliable
reconstruction of the unbounded rough surface, even for the case with a fairly large amount of noise
in the measured data.

To understand why the direct imaging method works, a theoretical analysis of the imaging method is presented.
To do this, we introduce Green's function $G(x,y)$ for the impedance half-plane (see \cite{CR96})
which plays an important role when analyzing the asymptotic behavior of the scattered field.
Further, an integral identity concerning the fundamental solution of the Helmholtz equation is established
for the unbounded rough surface, which is similar to the Helmholtz-Kirchhoff identity for bounded obstacles \cite{CCH13}.
In addition, a reciprocity relation is proved for the total field of the forward scattering problem.
Based on these results, the main results of the paper (Theorems \ref{thm307}-\ref{thm310}) are established
which lead to the required imaging function of the direct imaging method.

The remaining part of the paper is organized as follows.
Section \ref{sec2} gives a brief description of the scattering problems and introduces some notations and
inequalities that will be used in this paper.
Moreover, the well-posedness of the forward scattering problems, based on the integral equation method, will
be also presented without proof in this section.
In Section \ref{sec3}, we first conduct a theoretical analysis of the continuous imaging function
and then propose the direct imaging method for the inverse problem.
Finally, numerical examples are carried out in Section \ref{sec4}
to illustrate the effectiveness of the imaging method.

We conclude this section by introducing some notations used throughout the paper.
For $h\in\R$, define $U_h:=\{x\in\R^2~|~x_2>h\}$.
For $V\subset\R^n$ ($n=1,2$), denote by $BC(V)$ the set of bounded and continuous functions in $V$, a Banach space
under the norm $\|\psi\|_{\infty,V}:=\sup_{x\in V}|\psi(x)|$.
We write $\|\cdot\|_{\infty}$ for $\|\cdot\|_{\infty,\R^n}$.
For $0<\alpha\leq1$, denote by $BC^{0,\alpha}(V)$ the Banach space of functions $\phi\in BC(V)$ which are uniformly
H\"{o}lder continuous with exponent $\alpha$, equipped with the norm $\|\cdot\|_{0,\alpha,V}$ defined by
$\|\phi\|_{0,\alpha,V}:=\|\phi\|_\infty+\sup_{x,y\in V,x\neq y}[|\phi(x)-\phi(y)|/|x-y|^\alpha]$.
Given an open set $V\subset\R^2$ and $v\in L^\infty(V)$,
denote by $\pa_jv$ the (distributional) derivative $\pa v(x)/\pa x_j$, $j=1,2.$
Define $BC^{1,1}(V):=\{\varphi\in BC(V)~|~\pa_j\varphi\in BC(V),j=1,2\}$ with the norm
$\|\varphi\|_{1,V}:=\|\varphi\|_{\infty,V}+\|\pa_1\varphi\|_{\infty,V}+\|\pa_2\varphi\|_{\infty,V}$.
Finally, we introduce some spaces of smooth functions on the boundary $\G$.
Let $BC^{1,\al}(\G):=\{\varphi\in BC(\G)~|~\textrm{Grad}\,\varphi\in C^{0,\alpha}(\G)\}$
with the norm $\|\varphi\|_{1,\alpha,\G}:=\|\varphi\|_{\infty,\G}+\|\textrm{Grad}\,\varphi\|_{\infty,\G}$,
where $\textrm{Grad}$ denotes the surface gradient. For simplicity, we assume that
$f\in B(c_1,c_2):=\{f\in BC^{1,1}(\R)~|~f(s)\geq c_1,s\in\R\;\textrm{and}\;\|f\|_{1,1,\R}\leq c_2\}$
for some $c_1,c_2>0$.

\section{Problem Formulation}\label{sec2}
\setcounter{equation}{0}

In this section, we introduce the mathematical model of interest and propose
some existed results on the well-posedness for the forward scattering problems.
Some useful notations and inequalities used in the paper will also be presented.

\subsection{The forward scattering problems}

Given $f\in BC^{1,1}(\R)$ with $f_-:=\inf\limits_{x_1\in\R} f(x_1)>0,$
define the two-dimensional regions $D_+$ and $D_-$ by
\begin{align*}
D_+:=\left\{x=(x_1,x_2)\in\R^2 ~|~ x_2>f(x_1)\right\},\\
D_-:=\left\{x=(x_1,x_2)\in\R^2 ~|~ x_2<f(x_1)\right\},
\end{align*}
so that the unbounded rough surface $\Gamma$ is given by
\ben
\Gamma:=\left\{x=(x_1,f(x_1)) ~|~ x_1\in\R\right\}.
\enn

Let $u^i(x,y):=\varPhi_k(x,y)$ be an incident point source, where $\varPhi_k(x,y)$ is the fundamental solution
to the Helmholtz equation in two dimensions given by
\be\label{eq001}
\varPhi_k(x,y):=\frac{i}{4}H^{(1)}_0(k|x-y|), \quad x \not= y.
\en
Here, $H^{(1)}_0$ is the Hankel function of the first kind of order zero and $k$ is the wavenumber.
The forward scattering problem is to determine the unknown scattered wave $u^s$ in $D_+$ and
the unknown transmitted wave $u^t$ in $D_-$ such that the total field
$$
u:=\left\{
\begin{aligned}
&u^i + u^s &\quad &\text{in} \quad D_+,\\
&u^t &\quad &\text{in} \quad D_-
\end{aligned}
\right.
$$
satisfies the Helmholtz equation
\ben
\Delta u + k^2 u= 0  \quad \text{in} \quad \mathbb{R}^2\ba\Gamma
\enn
with $k=k_+>0$ in $D_+$ and $k=k_-> 0$ in $D_-$, respectively, where $k_+\not=k_-.$

We restrict our attention to the following three cases:

(1) the case when the total field vanishes on the boundary, so that the scattered field $u^s$,
the solution of the Helmholtz equation in $D_+$, satisfies the \emph{Dirichlet boundary condition}
$u^s = -u^i$ on $\Gamma$;

(2) the case when the total field satisfies the homogeneous \emph{impedance boundary condition}
${\pa_{\nu}u}-ik_+\rho u = 0$ on $\Gamma$, where $\nu(x)$ stands for the unit normal vector
at $x\in\Gamma$ pointing out of $D_+$ and $\pa_{\nu}u$ is the normal derivative of $u$;

(3) the case when the rough surface is penetrable and satisfies the \emph{transmission boundary condition}
$u^s+u^i=u^t, \pa_{\nu}u^s+\pa_{\nu}u^i=\pa_{\nu}u^t$ on $\Gamma.$

In order for the problem to have a unique solution, we adopt the so-called \emph{upward propagating
radiation condition (UPRC)} and \emph{downward propagating radiation condition (DPRC)} \cite{CZ98a,CZ98,CZ99}.
The scattered field $u^s$ is required to satisfy \emph{UPRC}:
for some $h_1>f_+:=\sup\limits_{x_1\in\mathbb{R}}f(x_1)$ and $\phi_1\in L^\infty(\Gamma_{h_1}),$
\be\label{eq206}
u^s(x)=2\int_{\Gamma_{h_1}}\frac{\pa\varPhi_{k_+}(x,y)}{\pa y_2}\phi_1(y)ds(y),\quad x\in U_{h_1}
\en
and the transmitted field $u^t$ is required to satisfy \emph{DPRC}:
for some $h_2<f_-$ and $\phi_2\in L^\infty(\Gamma_{h_2}),$
\be\label{eq207}
u^t(x)=-2\int_{\Gamma_{h_2}}\frac{\pa\varPhi_{k_-}(x,y)}{\pa y_2}\phi_2(y)ds(y),
\quad x\in\mathbb{R}^2\ba\ol{U}_{h_2},
\en
respectively. Here, $\varPhi_{k_\pm}$ is defined as $\varPhi$ with $k$ replaced by $k_\pm$.

The above scattering problems can now be formulated as the following boundary value problems for
the scattered field $u^s$ and the transmitted field $u^t$.

\textbf{\em Dirichlet scattering problem (DSP)}: Given $g\in BC(\Gamma)$,
determine $u^s\in C^2(D_+)\cap C(\ol{D}_+)$ such that

(i) $u^s$ is a solution of the Helmholtz equation
\be\label{eq204}
\Delta u^s + k_+^2 u^s = 0 \quad \text{in} \quad D_+,
\en

(ii) $u^s=g$ on $\Gamma$,

(iii) For some $\beta\in\R$,
\be\label{eq205}
\sup_{x \in D_+}x_2^\beta|u^s(x)|<\infty,
\en

(iv) $u^s$ satisfies the UPRC (\ref{eq206}).

Let $\mathscr{R}(D_+)$ denote the set of functions $w\in C^2(D_+)\cap C(\ol{D}_+)$
whose normal derivative defined by $\pa_{\nu}w(x):=\lim_{h\rightarrow0+}\nu(x)
\cdot\nabla w(x+h\nu(x))$ exists uniformly for $x$ on any compact subset of $\Gamma$.

\textbf{\em Impedance scattering problem (ISP)}: Given $g\in BC(\Gamma)$, $\rho\in BC(\Gamma)$,
determine $u^s\in \mathscr{R}(D_+)$ such that

(i) $u^s$ is a solution of the Helmholtz equation (\ref{eq204}) in $D_+$,

(ii) $\pa_{\nu}u^s-ik_+\rho u^s=g$ on $\Gamma$,

(iii) $u^s$ satisfies (\ref{eq205}) for some $\beta\in\R$,

(iv) $u^s$ satisfies the UPRC (\ref{eq206}),

(v) for some $\theta\in(0,1)$ and some constant $C_\theta>0$,
\ben
|\nabla u(x)|\leq C_\theta[x_2-f(x_1)]^{\theta-1}
\enn
for $x\in D_+\ba\ol{U}^+_b$ with $b=f_++1$.

The scattering problem by a penetrable rough surface can be formulated as follows.

\textbf{\em Transmission scattering problem (TSP)}: Let $\alpha\in(0,1)$, $h_1>f_+$ and $h_2<f_-$.
Given $g_1\in BC^{1,\alpha}(\Gamma)$ and $g_2\in BC^{0,\alpha}(\Gamma)$,
determine a pair of functions $(u^s,u^t)$ with $u^s\in C^2(D_+)\cap BC^1(\ol{D}_+\ba U_{h_1})$
and $u^t\in C^2(D_-)\cap BC^1(\ol{D}_-\cap\ol{U}_{h_2})$ such that

(i) $u^s$ is a solution of the Helmholtz equation (\ref{eq204}) in $D_+$
and $u^t$ is a solution of the Helmholtz equation
$\Delta u^t+k_-^2u^t=0$ in $D_-,$

(ii) $ u^s -u^t=g_1,~\pa_{\nu}u^s-\pa_{\nu}u^t=g_2\quad\textrm{on}\;\;\G$,

(iii) $u^s$ satisfies (\ref{eq205}) and $u^t$ satisfies
\ben
\sup_{x \in D_-}x_2^\beta|u^t(x)|<\infty
\enn
for some $\beta\in\R$,

(iv) $u^s$ satisfies the UPRC (\ref{eq206})
 and $u^t$ satisfies the DPRC (\ref{eq207}).

\subsection{Some useful notations}

In this subsection we introduce some basic notations and fundamental functions that will be needed in the
subsequent discussions. First, note that $H^{(1)'}_0=-H^{(1)}_1$ \cite{AS64} and that
by \cite[equation(9.2.7)]{AS64},
\be\label{eq317}
H^{(1,2)}_n(t)=\sqrt{\frac{2}{\pi t}}e^{\pm i(t-\frac{n \pi}{2}-\frac{\pi}{4})}
\left\{1+ O\left(\frac{1}{t}\right)\right\},\quad t\to\infty, \quad n=0,1,....
\en
This, combined with equation (\ref{eq001}), implies that
\be
\label{eq311}|\varPhi_k(x,y)|&\leq& C~|x-y|^{-1/2},\\
\label{eq312}\left|\frac{\pa\varPhi_k(x,y)}{\pa x_i}\right|&\leq& C~|x_i-y_i|~|x-y|^{-3/2}, \quad i=1,2
{\color{red},}\\
\label{eq2205}|\varPhi_k(x,y)-\varPhi_k(x,y')| &\leq& C~(1+|x_2|)(1+|y_2|)(|x-y|^{-3/2}+|x-y'|^{-3/2})
\en
for $|x-y|\geq\delta>0,~x=(x_1,x_2),~y=(y_1,y_2),~y'=(y_1,-y_2)$
with $C>0$ depending only on $k,\delta$.

In this paper, we also need the following \emph{impedance Green's function for}
the Helmholtz equation $(\Delta+k^2)u=0$ in the half-planes $U^{\pm}_a:=\{x=(x_1,x_2)\in\R^2~|~x_2\gtrless a\}$.
For any $k>0$ define
\ben
G_k^{\pm}(x,y;a):=\varPhi_k(x,y)+\varPhi_k(x,y'_a)+P^{\pm}_k(x-y'_a),\quad x,y\in U^\pm_a,
\enn
where
\ben
P^{\pm}_k(z):=\frac{|z|e^{ik|z|}}{\pi}
\int^\infty_0\frac{t^{-1/2}e^{-k|z|t}[|z|\pm z_2(1+it)]}{\sqrt{t-2i}[|z|t-i(|z|\pm z_2)]^2}dt,
\quad z\in \ol{U}^\pm_0
\enn
with the square root being taken so that $-\pi/2<\arg\sqrt{t-2i}<0$ and $y_a'= (y_1, 2a-y_2)$.
From \cite{CRZ99}, it is known that $P^{\pm}_k\in C(\ol{U}^\pm_0)\cap C^\infty(\ol{U}^\pm_0\ba\{0\})$
and $G_k^{\pm}(x,y;a)$ is a radiating solution in $U^\pm_a$ and satisfies the impedance boundary condition
$\pa_{\nu}P^\pm_k\pm ikP^\pm_k=0$ on $\G_a:=\{x=(x_1,x_2)~|~x_1\in\mathbb{R},~x_2=a\}$.
Further, it is shown in \cite[equation(2.14)]{NAC03} that
 \be
 \label{eq2200}|G_k^{\pm}(x,y;a)|&\leq& C~(1+|x_1-y_1|)^{-3/2}\\
 \label{eq2204}|\nabla_x G_k^{\pm}(x,y;a)| &\leq& C~(1+|x_1-y_1|)^{-3/2}\\
 \label{eq2201}|\nabla_y G_k^{\pm}(x,y;a)| &\leq& C~(1+|x_1-y_1|)^{-3/2}\\
 \label{eq2202}|\nabla_x\pa_{\nu(y)}G_k^{\pm}(x,y;a)|&\leq& C~(1+|x_1-y_1|)^{-3/2}
 \en
for $x\in\G_H,~y\in\G,~|x_2-y_2|\geq\delta>0$, with $C>0$ depending only on $a,k,\delta,\G$ and $H$.

We end this subsection by introducing certain layer potentials and boundary integral operators.
For $a<f_-$ define the single- and double-layer potentials: for $x\in U^+_a\ba\Gamma$
\ben
&&(\mathcal{S}^+_{k_{\color{red}+},a}\varphi)(x):=\int_\G G^+_{k_+}(x,y;a)\varphi(y)ds(y),\\
&&(\mathcal{D}^+_{k_{\color{red}+},a}\varphi)(x):=\int_\G\frac{\pa}{\pa\nu(y)}G^+_{k_+}(x,y;a)\varphi(y)ds(y)
\enn
and the boundary integral operators: for $x\in\Gamma$
\begin{align}
(S^+_{k_+,a}\varphi)(x)&:=\int_\Gamma G^+_{k_+}(x,y;a)\varphi(y)ds(y), \no\\
(K^+_{k_+,a}\varphi)(x)&:=\int_\Gamma\frac{\pa}{\pa\nu(y)}G^+_{k_+}\varphi(y)ds(y),\no\\
(K'^{+}_{k_+,a}\varphi)(x)&:=\int_\Gamma\frac{\pa}{\pa\nu(x)}G^+_{k_+}\varphi(y)ds(y),\no\\
(T^+_{k_+,a}\varphi)(x)&:=\frac{\pa}{\pa\nu(x)}\int_\Gamma\frac{\pa}{\pa\nu(y)}G^+_{k_+}\varphi(y)ds(y).\no
\end{align}
Further, for $a>f_+$ the layer-potential operators $\mathcal{S}^-_{k_-,a},\,\mathcal{D}^-_{k_-,a}$
for $x\in U^-_a\ba\Gamma$ and the boundary integral operators
$S^-_{k_-,a},\,K^-_{k_-,a},\,K'^{-}_{k_-,a},\,T^-_{k_-,a}$ for $x\in\Gamma$ can be defined similarly.

\subsection{Well-posedness of the forward scattering problems}

The well-posedness of the forward scattering problems described in Section \ref{sec2} has been studied
by using the variational and integral equation methods (see, e.g. \cite{CM05,CZ98,CZ99,NAC03,ZC03}).
In this subsection, we present these well-posedness results based on the integral equation method,
which will be needed in the remaining part of this paper.

\begin{theorem}\label{thm301}
(see \cite{CZ98}) Assume that $f\in B(c_1,c_2)$.
Then the problem (DSP) has exactly one solution in the form
\be\label{eq301}
u^s(x)=(\mathcal{D}^+_{k_+,0}\varphi)(x),\quad x\in D_+.
\en
Here, the density function $\varphi\in BC(\Gamma)$ is the unique solution to the boundary integral equation
\be\label{eq302}
A_D\varphi:=\left(-\frac{1}{2}I+K^+_{k_+,0}\right)\varphi=g,
\en
where the integral operator $A_D$ is bijective (and so boundedly invertible) in $BC(\Gamma)$.
Further, for each $g\in BC(\Gamma)$ we have the estimate
\be\label{eq303}
|u^s(x)|\leq C x_2^{1/2}\|g\|_{\infty,\Gamma}
\en
for some constant $C>0$ independent of $g.$
\end{theorem}

\begin{theorem}\label{thm302}
(see \cite{ZC03}) Assume that $f\in B(c_1,c_2)$ and $\rho\in B(c_1,c_2)$.
Then the problem (ISP) has exactly one solution in the form
\be\label{eq304}
u^s(x)=(\mathcal{S}^+_{k_+,0}\varphi)(x),\quad x\in D_+.
\en
Here, the density function $\varphi\in BC(\Gamma)$ is the unique solution to the boundary integral equation
\be\label{eq305}
A_I\varphi:=\left(\frac{1}{2}I+K'^{+}_{k_+,0}-ik_+\rho S^+_{k_+,0}\right)\varphi=g,
\en
where the integral operator $A_I$ is bijective (and so boundedly invertible) in $BC(\Gamma)$.
Moreover, for each $g\in BC(\Gamma)$ we have the estimate
\be\label{eq306}
|u^s(x)|\leq C x_2^{1/2}\|g\|_{\infty,\Gamma}
\en
for some constant $C>0$ independent of $g.$
\end{theorem}

\begin{theorem}\label{thm303}
(see \cite{CZ99,LSZR16,NAC03}) Given $g_1\in BC^{1,\alpha}(\Gamma),$ $g_2\in BC^{0,\alpha}(\Gamma)$
and $f\in B(c_1,c_2)$ and for $h\in\R$ with $-h<f_-<f_+<h$,
the problem (TSP) has exactly one solution $(u^s,u^t)$ in the form
\be\label{eq307}
&&u^s(x)=(\mathcal{D}^+_{k_+,-h}{\varphi_1})(x)+(\mathcal{S}^+_{k_+,-h}{\varphi_2})(x),\quad x\in D_+,\\ \label{eq308}
&&u^t(x)=(\mathcal{D}^-_{k_-,h}{\varphi_1})(x)+(\mathcal{S}^-_{k_-,h}{\varphi_2})(x),\quad x\in D_-.
\en
Here, $\varphi:=({\varphi_1},{\varphi_2})^T\in BC^{1,\al}(\G)\times BC^{0,\al}(\G)$
is the unique solution to the boundary integral equation
\be\label{eq309}
A_T\varphi=G
\en
with
\be\label{eq310}
A_T=\left(\begin{array}{cc}
K^+_{k_+,-h}-K^-_{k_-,h}+I&S^+_{k_+,-h}-S^-_{k_-,h}\\
T^+_{k_+,-h}-T^-_{k_-,h}&K'^{+}_{k_+,-h}-K'^{-}_{k_-,h}-I
\end{array}\right),\quad
G=\left(\begin{array}{c}
g_1\\
g_2
\end{array}\right),
\en
where the integral operator $A_T$ is bijective (and so boundedly invertible) in $[BC(\Gamma)]^2$.
Moreover, $u^s,u^t$ depend continuously on $\|g_1\|_{\infty,\Gamma}$ and $\|g_2\|_{\infty,\Gamma}$,
and $\nabla u^s,\nabla u^t$ depend continuously on $\|g\|_{1,\alpha,\Gamma}$
and $\|g_2\|_{0,\alpha,\Gamma}$.
\end{theorem}

\begin{remark}\label{wplp} {\rm
It has been shown in \cite{ACH02} that once the unique solvability of the boundary integral equations
(\ref{eq302}), (\ref{eq305}) or (\ref{eq309}) has been established in the space of bounded and continuous
functions, the unique solvability in $L^p$ $(1\leq p\leq\infty)$ can be obtained. This general result has been
applied in \cite{ACH02} to the integral equation formulation proposed in \cite{CZ98,CZ99,NAC03,ZC03}.
}
\end{remark}

\section{Direct imaging method for the inverse problem}\label{sec3}
\setcounter{equation}{0}

This section presents a direct imaging method to solve the inverse problem.
To do this, we first establish certain results for the forward scattering problems associated with
incident point sources.
In the following proofs, the constant $C>0$ may be different at different places.

\begin{lemma}\label{lem2301}
Assume that $(u^s,u^t)$ is the solution to the problem (TSP) with the boundary data $g=(g_1,g_2)^T$.
If $g\in [L^p(\Gamma)]^2$ with $1\leq p\leq\infty$, then $u^s,\pa_{\nu}u^s\in L^p(\G_H)$
for any $H>f_+$ and $u^t,\pa_{\nu}u^t\in L^p(\G_h)$ for any $h<f_-$.
\end{lemma}

\begin{proof}
From Theorem \ref{thm303} and Remark \ref{wplp}, the scattered field $u^s$ can be written in the form
\be\label{eq318}
u^s(x)=\int_\G\frac{\pa G^+_{k_+}(x,\xi)}{\pa \nu(\xi)}\varphi_1(\xi)ds(\xi)
+ \int_\G G^+_{k_+}(x,\xi)~\varphi_2(\xi)ds(\xi),\quad x\in D_+,
\en
where $(\varphi_1,\varphi_2)^T\in [L^p(\Gamma)]^2$ is the unique solution to the boundary integral
equation (\ref{eq309}) with $g\in [L^p(\Gamma)]^2$.
Define $\phi:=(\phi_1,\phi_2)^T\in [L^p(\mathbb{R})]^2$ by
\ben\label{eq2308}
\phi_1(s):=\varphi_1\left(\left(s,f(s)\right)\right),\quad
\phi_2(s):=\varphi_2\left(\left(s,f(s)\right)\right),\quad s\in\R.
\enn
It is obvious that $\|\phi\|_{[L^p(\mathbb{R})]^2}\leq\|\varphi\|_{[L^p(\Gamma)]^2}<\infty.$
Using (\ref{eq2200}) and (\ref{eq2201}), we can deduce that for $x\in\Gamma_H$,
\ben
\left|u^s(x)\right|&\leq&\int_\G\left|\frac{\pa G^+_{k_+}(x,\xi)}{\pa\nu(\xi)}\right||\varphi_1(\xi)|ds(\xi)
+\int_\G\left|G^+_{k_+}(x,\xi)\right||\varphi_2(\xi)|ds(\xi) \\
&\leq& C\int_\G\left(1+|x_1-\xi_1|\right)^{-3/2}\left(\left|\varphi_1(\xi)\right|
+\left|\varphi_2(\xi)\right|\right)ds(\xi) \\
&=& C\int_{-\infty}^{+\infty}\left(1+|x_1-s|\right)^{-3/2}~\left(\left|\phi_1(s)\right|
+\left|\phi_2(s)\right| \right)\sqrt{1+f'(s)^2}ds
\enn
with the constant $C>0$ depending only on $k,\delta$.
Since $(1+|\cdot|)^{-3/2} \in L^1(\mathbb{R}),~ \phi_1+\phi_2\in L^p(\mathbb{R})$ and
$f \in B(c_1,c_2)$, we have by Young's inequality that
\ben
\|u^s\|_{L^p(\Gamma_H)}=\|u^s(\cdot,H)\|_{L^p(\mathbb{R})}
\leq\left\|1+\left|\cdot\right|\right\|_{L^1(\mathbb{R})}\|\phi_1+\phi_2\|_{L^p(\mathbb{R})}<\infty,
\enn
that is, $u^s\in L^p(\Gamma_H)$.

Furthermore, by (\ref{eq318}) we have that for $x\in\Gamma_H$,
\ben
\pa_{\nu}u^s(x)&=&\frac{\pa}{\pa\nu(x)}\left\{\int_{\G}\frac{\pa G^+_{k_+}(x,\xi)}{\pa\nu(\xi)}\varphi_1(\xi)ds(\xi)
+ \int_{\G} G^+_{k_+}(x,\xi)\varphi_2(\xi)ds(\xi)\right\} \\
&=&\int_{\G}\frac{\pa^2 G^+_{k_+}(x,\xi)}{\pa\nu(x)\pa\nu(\xi)}\varphi_1(\xi)ds(\xi)
+\int_{\G}\frac{\pa G^+_{k_+}(x,\xi)}{\pa\nu(x)}\varphi_2(\xi)ds(\xi).
\enn
Using this representation and the inequalities (\ref{eq2204}) and (\ref{eq2202}) and arguing similarly as above,
we obtain that $\pa_{\nu}u^s\in L^p(\G_H)$.

The results for the transmitted field $u^t$ can be shown similarly. The lemma is thus proved.
\end{proof}

\begin{corollary}\label{asym}
For $y\in D_+$ let $u^s(\cdot,y)$ and $u^t(\cdot,y)$ be the scattered and transmitted fields, respectively,
by the penetrable rough surface $\G$ and generated by the incident point source $u^i(x,y)=\varPhi_{k_+}(x,y)$
located at $y$, then for any $\vep_1>0$ and $\vep_2>{1}/{3}$,
$u^s(\cdot,y)\in L^{2+\vep_1}(\G_H)$ and $\pa_{\nu}u^s(\cdot,y)\in L^{2/3+\vep_2}(\G_H)$ for any $H>f_+$
and $u^t(\cdot,y)\in L^{2+\vep_1}(\G_h)$ and $\pa_{\nu}u^t(\cdot,y)\in L^{2/3+\vep_2}(\G_h)$ for any $h<f_-$.
\end{corollary}

\begin{proof}
From (\ref{eq311}) and (\ref{eq312}) it follows that for any $\vep_1>0$,
\ben
\left(-\varPhi_{k_+}(\cdot,y),-\pa_{\nu}\varPhi_{k_+}(\cdot,y)\right)\in [L^{2+\vep_1}(\G)]^2,
\enn
so, by Lemma \ref{lem2301} we have \ben
u^s(\cdot,y)\in L^{2+\vep_1}(\G_H).
\enn
Now define
\be\label{equs}
\tilde{u}^i(x,y):=\varPhi_{k_+}(x,y)-\varPhi_{k_+}(x,y'),\quad
\tilde{u}^s(x,y):= u^s(x,y)+\varPhi_{k_+}(x,y').
\en
It is easy to see that $(\tilde{u}^s, u^t)$ is the solution to the problem (TSP) corresponding to
the incident wave $\tilde{u}^i$. From (\ref{eq312}) and (\ref{eq2205}) it follows that for any $\vep_2>0$,
\ben
\left(-\varPhi_{k_+}(\cdot,y)+\varPhi_{k_+}(\cdot,y'),
-\pa_{\nu}\varPhi_{k_+}(\cdot,y)+\pa_{\nu}\varPhi_{k_+}(\cdot,y')\right)
\in [L^{2/3+\vep_2}(\G)]^2.
\enn
Again, by Lemma \ref{lem2301} it is known that for any $\vep_2>1/3$,
\ben
\tilde{u}^s(\cdot,y)\in L^{2/3+\vep_2}(\G_H),\quad \pa_{\nu}\tilde{u}^s(\cdot,y)\in L^{2/3+\vep_2}(\G_H).
\enn
By this and (\ref{equs}), we have
\ben
\pa_{\nu}u^s(\cdot,y)=\pa_{\nu}\tilde{u}^s(\cdot,y)-\pa_{\nu}\varPhi_{k_+}(\cdot,y')\in L^{{2}/{3}+\vep_2}(\G_H).
\enn
The results for the transmitted field $u^t$ can be shown similarly.
This completes the proof.
\end{proof}

\begin{remark} \label{lemim} {\rm
With the help of Theorems \ref{thm301} and \ref{thm302} and Remark \ref{wplp},
Corollary \ref{asym} can also be extended to the cases of an impenetrable rough surface.
}
\end{remark}

The following lemma is similar to the Helmholtz-Kirchhoff identity which plays an important role
in the case of inverse scattering by bounded obstacles \cite[Lemma 3.1]{CCH13}.

\begin{lemma}\label{lem305}
For any $H\in\mathbb{R}$ we have
\ben
&&\int_{\G_H}\left(\frac{\pa\varPhi_k(x,y)}{\pa\nu(x)}\ol{\varPhi_k(x,z)}
-\varPhi_k(x,y)\frac{\pa\ol{\varPhi_k(x,z)}}{\pa\nu(x)}\right)ds(x)\\
&&\qquad=\frac{i}{4\pi}\int_{\Sp^1_+}e^{ik\hat x\cdot(z-y)}ds(\hat x),\quad y,z\in\R^2\ba\ol{U}_H,
\enn
where $\nu$ denotes the unit normal vector on $\G_H$ pointing into $U_H$.
\end{lemma}

\begin{proof}
Let $\pa B^+_R$ be the half circle above $\Gamma_H$ centered at $(0,H)$ and with radius $R>0$ and
define $\Gamma_{H,R}:=\{x\in\Gamma_H\;|\;|x_1|\leq R\}$.
Denote by $\Om$ the bounded region enclosed by $\Gamma_{H,R}$ and $\pa B^+_R$.
Using Green's theorem in $\Omega$, we obtain that
\begin{align}\label{eq100}
0&=\int_{\Omega}\left(\Delta\varPhi_k(x,y)+k^2\varPhi_k(x,y)\right)\ol{\varPhi_k(x,z)}dx \no\\
&=\int_{\Gamma_{H,R}}\left(-\frac{\pa\varPhi_k(x ,y)}{\pa\nu(x)}\ol{\varPhi_k(x,z)}
+\varPhi_k(x,y)\frac{\pa\ol{\varPhi_k(x,z)}}{\pa\nu(x)}\right)ds(x) \nonumber\\
&+\int_{\pa B^+_R}\left(\frac{\pa\varPhi_k(x ,y)}{\pa\nu(x)}\ol{\varPhi_k(x,z)}
-\varPhi_k(x,y)\frac{\pa\ol{\varPhi_k(x,z)}}{\pa\nu(x)}\right)ds(x).
\end{align}
From (\ref{eq001}) and (\ref{eq317}) we know that
\ben
\varPhi_k(x,y)=\frac{e^{ik|x|}}{\sqrt{|x|}}\frac{e^{i\pi/4}}{\sqrt{8\pi k}}e^{-ik\hat x\cdot y}
+O\left(\frac1{|x|}\right),\quad |x|\to\infty.
\enn
By this and the Sommerfeld radiation condition
\ben
\frac{\pa\varPhi_k(x,y)}{\pa\nu(x)}-ik\varPhi_k(x,y)=O(|x|^{-3/2}),\quad |x|\to\infty,
\enn
we obtain on letting $R\to\infty$ in (\ref{eq100}) that
\ben
&&\int_{\Gamma_H}\left(\frac{\pa\varPhi_k(x ,y)}{\pa\nu(x)}\ol{\varPhi_k(x,z)}
-\varPhi_k(x, y)\frac{\pa\ol{\varPhi_k(x,z)}}{\pa\nu(x)}\right)ds(x)\\
&&=\lim_{R\to\infty}\int_{\pa B^+_R}\left(\frac{\pa\varPhi_k(x ,y)}{\pa\nu(x)}\ol{\varPhi_k(x,z)}
-\varPhi_k(x,y)\frac{\pa\ol{\varPhi_k(x,z)}}{\pa\nu(x)}\right)ds(x)\\
&&=\lim_{R\to\infty}2ik\int_{\pa B^+_R}\varPhi_k(x,y)\ol{\varPhi_k(x,z)} ds(x)\\
&&=\lim_{R\to\infty}\frac{i}{4\pi}\int_{\pa B^+_R}\frac{1}{|x|}e^{ik\hat{x}\cdot(z-y)}ds(x)\\
&&=\frac{i}{4\pi}\int_{\mathbb{S}^1_+}e^{ik\hat{x}\cdot(z-y)}ds(\hat{x}).
\enn
This completes the proof.
\end{proof}

We also need the reciprocity relation $u(x,y)=u(y,x)$ for an unbounded rough surface.
The reciprocity relation can be found in \cite[Chapter 3.3]{CK} for the case of bounded obstacles.
For a locally rough surface, since the scattered field $u^s$ satisfies the Sommerfeld radiation
condition \cite{BL11,W87}, the reciprocity relation can be proved similarly as for the case of
bounded obstacles.
For the case of a global unbounded rough surface, the reciprocity relation has been established
in \cite[Theorem3.14]{L03} by using the assumption that the scattered field generated by a point
source satisfies the Sommerfeld radiation condition. However, there is no rigorous proof for this
assumption in \cite{CRZ98,L03}. Here, we give a proof of the reciprocity relation for the case of
a globally rough surface without this assumption.

\begin{lemma}\label{lem306} (Reciprocity relation) Let $u$ denote the total field in $\R^2\ba\G$
generating by a penetrable rough surface $\G$ corresponding to the incident point source
$u^i=\varPhi_{k_+}(x,y),$ that is,
$$
u=\left\{\begin{aligned}
&u^i + u^s &\quad\text{in}\quad D_+,\\
&u^t &\quad \text{in}\quad D_-.
\end{aligned}
\right.
$$
Then $u(x,y)=u(y,x),~x,y\in\mathbb{R}^2\ba\Gamma$.
\end{lemma}

\begin{proof}
For $z\in\R^2$, let $z=(z_1,z_2)$.
For $b,L,\vep>0$ define $D_{b,L,\vep}:=\{z\in\R^2\;|\;|z_2|< b,|z_1|<L,|z-x|>\vep,|z-y|>\vep\}$.
Choose $\vep$ sufficiently small and $b,L$ large enough such that
$\ol{B_{\vep}(x)}\subset D_{b,L,\vep},\ol{B_{\vep}(y)}\subset D_{b,L,\vep}$ and
$\ol{B_{\vep}(x)}\cap\ol{B_{\vep}(y)}=\emptyset$.
Then, since, by Theorem \ref{thm303}, $u(\cdot,y),~u(\cdot,x)\in C^2(D_{b,L,\vep})\cap C(\ol{D}_{b,L,\vep})$,
we can apply Green's second theorem in $D_{b,L,\varepsilon}$ to give that
\be\label{eq330}
0=\int_{\pa D_{b,L,\vep}}\left(\frac{\pa u(z,y)}{\pa\nu(z)}u(z,x)
-\frac{\pa u(z,x)}{\pa\nu(z)}u(z,y)\right)ds(z).
\en

For any $v(z)\in L^2(\Gamma_b)$ with $b\in\R$, let $\mathscr{F}v$ denote the Fourier transform of
$v$ with respect to $z_1$, that is,
\ben
(\mathscr{F}v)(\xi,z_2)\big|_{z_2=b}=\int^{+\infty}_{-\infty}e^{-i z_1\xi}v(z)d z_1.
\enn
Then, by \cite[(2.7)]{CZ98} UPRC can be rewritten in the angular spectrum representation
\be\label{eq331}
u(z)=\frac{1}{2\pi}\int_{-\infty}^{+\infty}e^{i(z_2-h)\sqrt{k_+^2-\xi^2}+iz_1\xi}\hat u_h(\xi)d\xi,\quad z_2>h,
\en
where $\hat{u}_h(\xi):=(\mathscr{F}u)(\xi,h)$.
It follows from the representation (\ref{eq331}) that
\be\label{fourier1}
(\mathscr{F}u)(\xi,z_2)\big|_{z_2=b}&=&e^{i(b-h)\sqrt{k_+^2-\xi^2}}\hat u_h(\xi),\quad b>h,\\
(\mathscr{F}\pa_{\nu}u)(\xi,z_2)\big|_{z_2=b}&=&\frac{\pa}{\pa z_2}(\mathscr{F}u)(\xi,z_2)\big|_{z_2=b}\no\\
\label{fourier2}
&=&i\sqrt{k_+^2-\xi^2}e^{i(b-h)\sqrt{k_+^2-\xi^2}}\hat u_h(\xi),\quad b>h,
\en
where $\nu$ denotes the unit normal vector on $\Gamma_b$ pointing into $U_b$.
Thus, by (\ref{fourier1}), (\ref{fourier2}) and Parseval's formula we have
\be\label{eq2}
&&\int_{\G_b}\left(\frac{\pa u(z,y)}{\pa\nu(z)}u(z,x)-\frac{\pa u(z,x)}{\pa\nu(z)}u(z,y)\right)ds(z)\no\\
&&=\frac{1}{2\pi}\int_{-\infty}^{+\infty}\left( (\mathscr{F}\pa_{\nu}u)(\xi,z_2;y)\big|_{z_2=b}\ol{(\mathscr{F}\ol{u})}(\xi,b;x)\right.\no\\
&&\qquad\qquad\qquad\qquad\left.-(\mathscr{F} u)(\xi,b;x)\ol{(\mathscr{F}
     \ol{\pa_{\nu}u})}(\xi,z_2;y)\big|_{z_2=b}\right)d\xi\no\\
&&=\frac{1}{2\pi}\int_{-\infty}^{+\infty}\left( {i\sqrt{k_+^2-\xi^2}}e^{i(b-h)\sqrt{k_+^2-\xi^2}}\hat u_h(\xi;y)
e^{-i(b-h)\sqrt{k_+^2-\xi^2}}\ol{\hat {\ol{u}}_h(\xi;x)}\right.\no\\
&&\qquad\qquad\qquad\qquad\left.-{i\sqrt{k_+^2-\xi^2}}e^{i(b-h)\sqrt{k_+^2-\xi^2}}\hat u_h(\xi;x)
e^{-i(b-h)\sqrt{k_+^2-\xi^2}}\ol{\hat {\ol{u}}_h(\xi;y)}\right)d\xi\no\\
&&=\frac{1}{2\pi}\int_{-\infty}^{+\infty} i\sqrt{k_+^2-\xi^2}
  \left((\mathscr{F}u)(\xi,h;y)\ol{(\mathscr{F}\ol{u})}(\xi,h;x)\right.\no\\
&&\qquad\qquad\qquad\qquad\qquad\left.-(\mathscr{F}u)(\xi,h;x)\ol{(\mathscr{F}\ol{u})}(\xi,h;y)\right)d\xi.
\en
In the above equation $x$ and $y$ are used to indicate the dependence on the locations of the incident point sources.
Since $\ol{(\mathscr{F}\ol{u})}(\xi)=(\mathscr{F}u)(-\xi)$, we get
\be\label{eq3}
&&\int_{-\infty}^{+\infty}i\sqrt{k_+^2-\xi^2}\left((\mathscr{F}u)(\xi,h;y)\ol{(\mathscr{F}\ol{u})}(\xi,h;x)
  -(\mathscr{F}u)(\xi,h;x)\ol{(\mathscr{F}\ol{u})}(\xi,h;y)\right)d\xi \no\\
=&&\int_{-\infty}^{+\infty}i\sqrt{k_+^2-\xi^2}\left((\mathscr{F}u)(\xi,h;y)(\mathscr{F}u)(-\xi,h;x)\right.\no\\
&&\qquad\qquad\qquad\qquad\qquad-\left.(\mathscr{F}u)(\xi,h;x)(\mathscr{F}u)(-\xi,h;y)\right)d\xi.
\en
Noting that $a(\xi)=i\sqrt{k_+^2-\xi^2}$ is an even function of $\xi$, we obtain that the right-hand side
of (\ref{eq3}) vanishes. This, together with (\ref{eq2}), implies that
\be\label{eq332}
\int_{\Gamma_b}\left(\frac{\pa u(z,y)}{\pa\nu(z)}u(z,x)
-\frac{\pa u(z,x)}{\pa\nu(z)}u(z,y)\right)ds(z)=0.
\en
Similarly, we can show that
\ben
\int_{\Gamma_{-b}}\left(\frac{\pa u(z,y)}{\pa\nu(z)}u(z,x)
-\frac{\pa u(z,x)}{\pa\nu(z)}u(z,y)\right)ds(z)=0.
\enn
Using \cite[Theorem 5.1]{CRZ98}, we derive that
\be\label{eq333}
\int_{\Gamma_{\pm L}}\left(\frac{\pa u(z,y)}{\pa\nu(z)}u(z,x)
-\frac{\pa u(z,x)}{\pa\nu(z)}u(z,y)\right)ds(z)=0,\qquad L\rightarrow\infty.
\en
Applying Green's second theorem to $u^s(z,y)$ and $u(z,x)$ in $B_{\vep}(y)$ gives
\be\label{eq335}
0=\int_{\pa B_{\varepsilon}(y)}\left(\frac{\pa u^s(z,y)}{\pa\nu(z)}u(z,x)
-\frac{\pa u(z,x)}{\pa\nu(z)}u^s(z,y)\right)ds(z).
\en
Then applying Green's second theorem to $u^s(z,x)$ and $u(z,y)$ in $B_{\vep}(x)$ gives
\be\label{eq336}
0=\int_{\pa B_{\vep}(x)}\left(\frac{\pa u(z,y)}{\pa\nu(z)}u^s(z,x)
-\frac{\pa u^s(z,x)}{\pa\nu(z)}u(z,y)\right)ds(z).
\en
Now subtracting (\ref{eq332})-(\ref{eq336}) into (\ref{eq330}) and using
Green's representation theorem \cite[(2.5)]{CK} yield
\ben
0&=&\int_{\pa B_{\vep}(y)}\left(\frac{\pa\varPhi_{k_+}(z,y)}{\pa\nu(z)}u(z,x)-
\frac{\pa u(z,x)}{\pa\nu(z)}\varPhi_{k_+}(z,y)\right) ds(z) \\
&&+ \int_{\pa B_{\vep}(x)}\left(\frac{\pa u(z,y)}{\pa\nu(z)}\varPhi_{k_+}(z,x)
-\frac{\pa\varPhi_{k_+}(z,x)}{\pa\nu(z)}u(z,y)\right)ds(z)
=u(y,x)-u(x,y).
\enn
The proof is thus complete.
\end{proof}

\begin{remark} {\rm
The reciprocity relation is also valid for the cases of impenetrable rough surfaces.
}
\end{remark}

\subsection{The penetrable rough surface}

We now prove the following theorem for the penetrable rough surface
which leads to the imaging function required for the imaging method.

\begin{theorem}\label{thm307}
Assume that $(u^s(x,y),u^t(x,y))$ is the solution to the problem (TSP) with the boundary data
$g(x,y)=(-\varPhi_{k_+}(x,y),-{\pa\varPhi_{k_+}(x,y)}/{\pa\nu(x)})^T$, $x\in\Gamma$.
For $z\in U^+_h\ba\ol{U}{}^+_H$ with $H>f_+$ and $h<f_-$ define
\be\no
&&U^s(y,z)=\int_{\G_H}\left(\frac{\pa u^s(x,y)}{\pa\nu(x)}\ol{\varPhi_{k_+}(x,z)}
-u^s(x,y)\frac{\pa\ol{\varPhi_{k_+}(x,z)}}{\pa\nu(x)}\right)ds(x)\\ \label{eq2318}
&&\qquad\qquad-\frac{i}{4\pi}\int_{\Sp^1_-}e^{ik_+\hat{x}\cdot(y'-z')}ds(\hat{x}),\; y\in D_+, \\ \label{eq2319}
&&U^t(y,z)=\int_{\G_h}\left(\frac{\pa u^t(x,y)}{\pa\nu(x)}\ol{\varPhi_{k_+}(x,z)}
-u^t(x,y)\frac{\pa\ol{\varPhi_{k_+}(x,z)}}{\pa\nu(x)}\right)ds(x),\; y\in D_-.\;\;
\en
Then $(U^s(y,z),U^t(y,z))$ is well-defined and solves the problem (TSP) with the boundary data
\ben
G(y,z)=\left(-\frac{i}{2}J_0(k_+|y-z|),-\frac{i}{2}\frac{\pa J_0(k_+|y-z|)}{\pa\nu(y)}\right)^T,\quad y\in\G.
\enn
\end{theorem}

\begin{proof}
We first prove that $(U^s(y,z),U^t(y,z))$ is well-defined.
By Corollary \ref{asym} it follows that for any $\vep_1>0,~\vep_2>{1}/{3}$,
\ben
u^s(\cdot,y)\in L^{2+\vep_1}(\G_H),\quad\pa_{\nu}u^s(\cdot,y)\in L^{2/3+\vep_2}(\G_H).
\enn
Further, by (\ref{eq311}) and (\ref{eq312}) we have
\ben
\varPhi_{k_+}(\cdot,z)\in L^{2+\vep_1}(\G_H),\quad\pa_{\nu}\varPhi_{k_+}(\cdot,z)\in L^{2/3+\vep_2}(\G_H).
\enn
Choose $\vep_1>0,~\vep_2>1/3$ such that $1/(2+\vep_1)+1/(\vep_2+2/3)=1.$ Then,
by H\"{o}lder's inequality it is obtained that
$\|u^s(\cdot,y)\pa_{\nu}\varPhi_{k_+}(\cdot,z)\|_{L^1(\G_H)}<\infty,$
$\|\pa_{\nu}u^s(\cdot,y)\varPhi_{k_+}(\cdot,z)\|_{L^1(\G_H)}<\infty.$
This implies that $U^s(y,z)$ is well-defined. It can be shown similarly that $U^t(y,z)$ is well-defined.

Since $(u^s(y,x),u^t(y,x))$ is a solution to the problem (TSP) with the boundary data
$g(y,x)=(-\varPhi_{k_+}(y,x),-{\pa\varPhi_{k_+}(y,x)}/{\pa\nu(y)})^T,$ $y\in\G$, Then, by Theorem \ref{thm303}
there exists $\varphi_x(y):=({\varphi_{1,x}(y)},\varphi_{2,x}(y))^T\in BC^{1,\al}(\G)\times BC^{0,\al}(\G)$
such that
\be\label{0us}
u^s(y,x)&=&(\mathcal{D}^+_{k_+,-h}{\varphi_{1,x}})(y)+(\mathcal{S}^+_{k_+,-h}
{\varphi_{2,x}})(y),\quad y\in D_+,\\ \label{0ut}
u^t(y,x)&=&(\mathcal{D}^-_{k_-,h}{\varphi_{1,x}})(y)+(\mathcal{S}^-_{k_-,h}
{\varphi_{2,x}})(y),\quad y\in D_-,
\en
where $\varphi_x(y)$ satisfies the integral equation $A_T\varphi_x=g(\cdot,x)$ with $A_T$ given in (\ref{eq310})
and the integral operator $A_T$ is bijective (and so boundedly invertible)
in $[BC(\G)]^2$. Here, we use the subscript $x$ to indicate the dependance on the point $x$.
Further, since the boundary data $g(y,x)$ is differentiable for $x\in\G_H$ and $y\in\G$,
we obtain that $\varphi_x$ is differentiable with respect to $x_2$ with
$\pa\varphi_x/\pa x_2= A_T^{-1}{\pa g(\cdot,x)}/{\pa x_2}$.
Now define $\tilde{\varphi}_{x}:=\pa\varphi_x/\pa x_2$, $\tilde{g}(\cdot,x):=\pa g(\cdot,x)/{\pa x_2}$
and
\be\label{1us}
\tilde{u}^s(y,x)&:=&(\mathcal{D}^+_{k_+,-h}{\tilde{\varphi}_{1,x}})(y)
+(\mathcal{S}^+_{k_+,-h}{\tilde{\varphi}_{2,x}})(y),\quad y\in D_+,\\ \label{1ut}
\tilde{u}^t(y,x)&:=&(\mathcal{D}^-_{k_-,h}{\tilde{\varphi}_{1,x}})(y)
+(\mathcal{S}^-_{k_-,h}{\tilde{\varphi}_{2,x}})(y),\quad y\in D_-.
\en
Noting that $A_T\tilde{\varphi_x}=\tilde{g}(\cdot,x)$, we know that
$(\tilde{u}^s(\cdot,x),\tilde{u}^t(\cdot,x))$ is a solution to the problem (TSP) with the boundary data
$\tilde{g}(\cdot,x)$. Since $\mathcal{D}^+_{k_+,-h}$ and $\mathcal{S}^+_{k_+,-h}$ in (\ref{1us})
are bounded linear operators (see \cite{CR96,NAC03,ZC03}), it follows that
\be\label{eq008}
\frac{\pa u^s(\cdot,x)}{\pa x_2}=\tilde{u}^s(\cdot,x).
\en
Similarly, it can be proved that
\be\label{eq005}
\frac{\pa u^t(\cdot,x)}{\pa x_2}= \tilde{u}^t(\cdot,x).
\en
By the above two equations and the reciprocity relation in Lemma \ref{lem306}, we know that
$({\pa u^s(\cdot,x)}/{\pa x_2},{\pa u^t(\cdot,x)}/{\pa x_2})$ is the solution to
the problem (TSP) with the boundary data ${\pa g(y,x)}/{\pa x_2}$, $y\in\G.$

Define
\ben
\tilde{\tilde{\varphi}}_{z}(y)&:=&\int_{\G_H}\left(\tilde{\varphi}_x(y)\ol{\varPhi_{k_+}(x,z)}
 -\varphi_x(y)\frac{\pa\ol{\varPhi_{k_+}(x,z)}}{\pa\nu(x)}\right)ds(x),\quad y\in\G,\\
\tilde{\tilde{g}}(y,z)&:=&\int_{\G_H}\left(\tilde{g}(y,x)\ol{\varPhi_{k_+}(x,z)}
  -g(y,x)\frac{\pa \ol{\varPhi_{k_+}(x,z)}}{\pa\nu(x)}\right)ds(x),\quad y\in\G,\\
\tilde{\tilde{u}}^s(y,z)&:=&\int_{\G_H}\left(\frac{\pa u^s(x,y)}{\pa x_2}\ol{\varPhi_{k_+}(x,z)}
-u^s(x,y)\frac{\pa\ol{\varPhi_{k_+}(x,z)}}{\pa\nu(x)}\right)ds(x),\quad y\in D_+,\\
\tilde{\tilde{u}}^t(y,z)&:=&\int_{\Gamma_H}\left(\frac{\pa u^t(x,y)}{\pa x_2}\ol{\varPhi_{k_+}(x,z)}
-u^t(x,y)\frac{\pa\ol{\varPhi_{k_+}(x,z)}}{\pa\nu(x)}\right)ds(x),\quad y\in D_-.
\enn
By (\ref{eq311}), (\ref{eq312}) and Remark \ref{wplp}, and since
\be\label{eq011}
A_T\tilde{\tilde{\varphi}}_z(y)=\tilde{\tilde{g}}(y,z),\quad y\in\G,
\en
we can show that $\tilde{\tilde{g}}(y,z)$ and $\tilde{\tilde{\varphi}}_{z}(\cdot)$ are well-defined.
Further, by (\ref{0us}) and (\ref{1us}) we know that
\be \label{eq012}
\tilde{\tilde{u}}^s(y,z)=(\mathcal{D}^+_{k_+,-h}\tilde{\tilde{\varphi}}_{1,z})(y)
+(\mathcal{S}^+_{k_+,-h}\tilde{\tilde{\varphi}}_{2,z})(y).
\en
Similarly, we can prove that
\be\label{eq013}
\tilde{\tilde{u}}^t(y,z)=(\mathcal{D}^-_{k_-,h}\tilde{\tilde{\varphi}}_{1,z})(y)
+(\mathcal{S}^-_{k_-,h}\tilde{\tilde{\varphi}}_{2,z})(y).
\en
From (\ref{eq011}), (\ref{eq012}) and (\ref{eq013}), we know that
$(\tilde{\tilde{u}}^s(y,z),\tilde{\tilde{u}}^t(y,z))$ is the solution to the problem (TSP)
with the boundary data $\tilde{\tilde{g}}(y,z)$, $y\in\G.$
Thus, $(U^s(y,z),U^t(y,z))$ is a solution to the problem (TSP) with the boundary data
$G(y,z)=(h(y,z),{\pa h(y,z)}/{\pa\nu(y)})^T$, where
\ben
h(y,z)&=&-\int_{\G_H}\left(\frac{\pa\varPhi_{k_+}(x,y)}{\pa\nu(x)}\ol{\varPhi_{k_+}(x,z)}
-\varPhi_{k_+}(x,y)\frac{\pa\ol{\varPhi_{k_+}(x,z)}}{\pa\nu(x)}\right)ds(x)\\
&&-\frac{i}{4\pi}\int_{\mathbb{S}^1_-}e^{ik_+\hat{x}\cdot(y'-z')}ds(\hat{x}),\quad y\in\G.
\enn
By using Lemma \ref{lem305} and the Funk-Hecke formula \cite{P10}, it follows that
\ben
h(y,z)=-\frac{i}{4\pi}\int_{\mathbb{S}^1_+}e^{ik_+\hat{x}\cdot(z-y)}ds(\hat{x})
-\frac{i}{4\pi}\int_{\mathbb{S}^1_-}e^{ik_+\hat{x}\cdot(y'-z')}ds(\hat{x})
=-\frac{i}{2}J_0(k_+|y-z|).
\enn
The proof is thus complete.
\end{proof}

\begin{remark}\label{rmbehavior} {\rm
From Theorems \ref{thm303} and \ref{thm307}, we know that
\be\label{eq313}
&&U^s(y,z)=(\mathcal{D}^+_{k_+,-h}{\psi_{1,z}})(y)
           +(\mathcal{S}^+_{k_+,-h}{\psi_{2,z}})(y),\quad x\in D_+,\\ \label{eq314}
&&U^t(y,z)=(\mathcal{D}^-_{k_-,h}{\psi_{1,z}})(y)
           +(\mathcal{S}^-_{k_-,h}{\psi_{2,z}})(y),\quad x\in D_-,
\en
where $\psi_z:=(\psi_{1,z},\psi_{2,z})^T$ is the unique solution to the integral equation $A_T\psi_z=G$ with
$$
G(y,z)=\left(-\frac{i}{2}J_0(k_+|y-z|),-\frac{i}{2}\frac{\pa J_0(k_+|y-z|)}{\pa\nu(y)}\right)^T
=:\left(g_{1,z}(y),g_{2,z}(y)\right)^T.
$$
Here, we use the subscript $z$ to indicate the dependence on the point $z$.
Further, since $A_T$ is bijective (and so boundedly invertible) in $[BC(\Gamma)]^2$, we have
\be\no
C_1\left(\|g_{1,z}\|_{\infty,\G}+\|g_{2,z}\|_{\infty,\G}\right)
&\le&\|{\psi}_{1,z}\|_{\infty,\G}+\|{\psi}_{2,z}\|_{\infty,\G}\\ \label{eq339}
&\le& C_2\left(\|g_{1,z}\|_{\infty,\G}+\|g_{2,z}\|_{\infty,\G}\right)
\en
for some positive constants $C_1,C_2$. Note that the Bessel functions $J_0$ and $J_1$ have
the following behavior \cite[Section 3.4]{CK} (see also Figure \ref{fig10}): For $n=0,1,2,...$
\ben
J_n(t)&=&\sum_{p=0}^\infty\frac{(-1)^p}{p!(n+p)!}\left(\frac{t}{2}\right)^{n+2p},\quad t\in\R\\
J_n(t)&=&\sqrt{\frac{2}{\pi t}}\cos\left(t-\frac{n\pi}{2}
-\frac{\pi}{4}\right)\left\{1+O\left(\frac{1}{t}\right)\right\},\quad t\rightarrow\infty
\enn
and $J'_0(t)=-J_1(t)$.
\begin{figure}[htbp]
  \centering
  \subfigure{\includegraphics[width=2in]{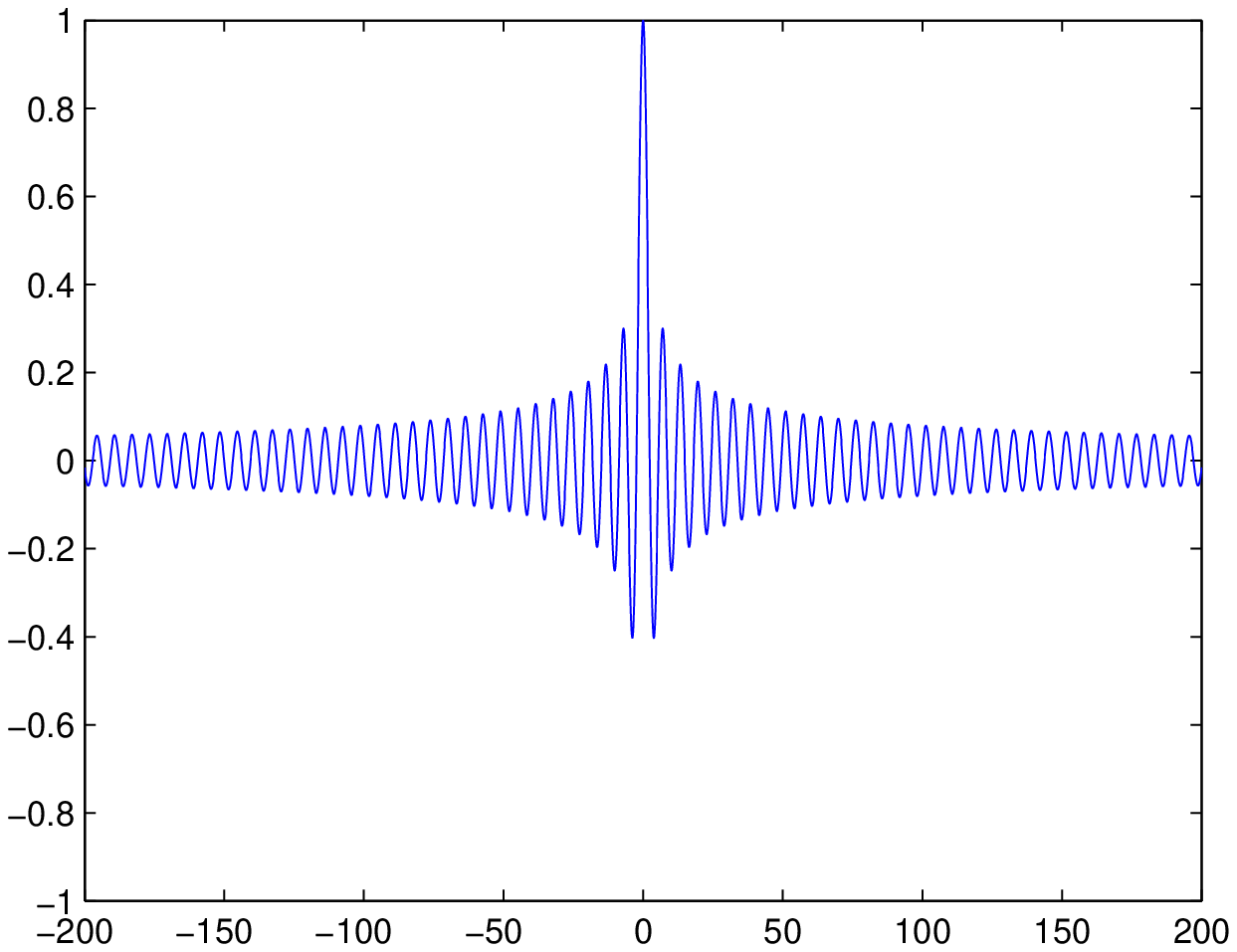}}
  \subfigure{\includegraphics[width=2in]{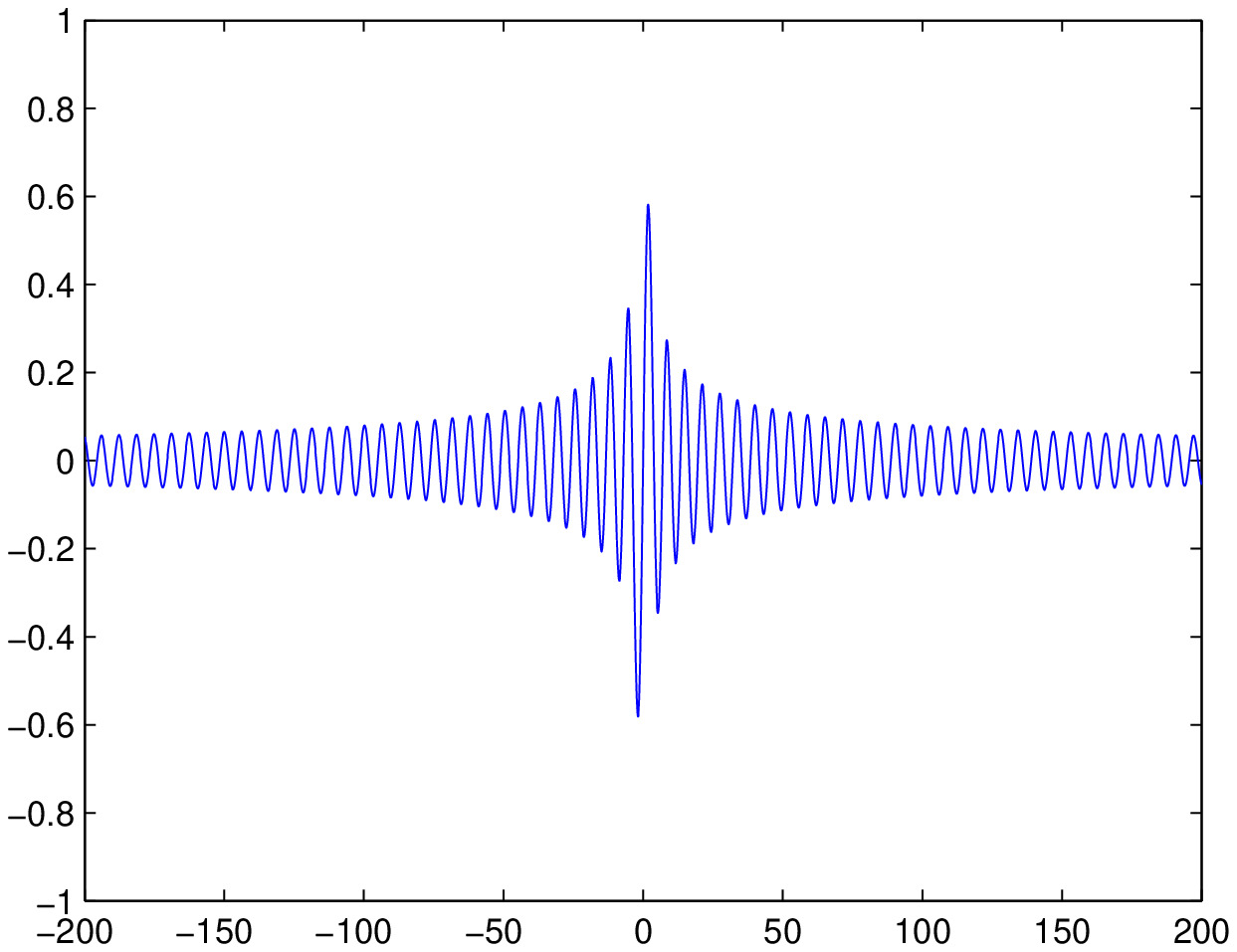}}
\caption{The behavior of the Bessel function $J_0$ (left) and $J_1$ (right).
}\label{fig10}
\end{figure}
Thus
\ben
g_{1,z}(y)&=&\left\{
\begin{array}{ll}
-{i}/{2} &\textrm{if}\;y=z,\\
O\left({|y-z|^{-1/2}}\right)&\textrm{if}\;|y-z|>>1, \end{array} \right.\\
g_{2,z}(y)&=&\left\{
\begin{array}{ll}
0 &\textrm{if}\;y=z,\\
O\left({|y-z|^{-1/2}}\right)&\textrm{if}\;|y-z|>>1.
\end{array} \right.
\enn
By (\ref{eq339}), we have
\ben
\left\{\begin{array}{ll}
\|{{\psi}_{1,z}}\|_{\infty,\Gamma}+\|{{\psi}_{2,z}}\|_{\infty,\Gamma}
\geq ({1}/{2})C_1 &\textrm{if}\;z\in\Gamma,\\
\|{{\psi}_{1,z}}\|_{\infty,\Gamma}+\|\psi_{2,z}\|_{\infty,\Gamma}
=O\left({d(z,\Gamma)^{-1/2}}\right)&\textrm{if}\;d(z,\Gamma)>>1. \end{array} \right.
\enn
From this and (\ref{eq313}), it is expected that the scattered field $U^s(y,z)$
takes a large value when $z\in\Gamma$ and decays as $z$ moves away from $\Gamma$.
}
\end{remark}

\subsection{The impenetrable rough surfaces}

In this subsection, we briefly present some results for the cases of an impenetrable rough surface
similar to Theorem \ref{thm307} which can be shown similarly.

\begin{theorem}\label{thm309}
Assume that $u^s(x,y),$ $x,y\in D_+,$ is the solution to the problem (DSP)
with the boundary data $g(x,y)=-\varPhi_{k_+}(x,y),x\in\G$.
For $z\in U^+_h\ba\ol{U}{}^+_H$ with $H>f_+$ and $h<f_-$ define
\be\no
&&U^s(y,z)=\int_{\G_H}\left(\frac{\pa u^s(x,y)}{\pa\nu(x)}\ol{\varPhi_{k_+}(x,z)}
-u^s(x,y)\frac{\pa\ol{\varPhi_{k_+}(x,z)}}{\pa\nu(x)}\right)ds(x)\\ \label{eq2320}
&&\qquad\qquad-\frac{i}{4\pi}\int_{\Sp^1_-}e^{ik_+\hat{x}\cdot(y'-z')}ds(\hat{x}),\quad y\in D_+
\en
which is independent of $h$. Then $U^s(y,z)$ is well-defined and solves the problem (DSP)
with the boundary data
\ben
G(y,z)=-({i}/{2})J_0(k_+|y-z|),\quad y\in\G.
\enn
\end{theorem}

\begin{theorem}\label{thm310}
Assume that $u^s(x,y),$ $x,y\in D_+,$ is the solution to the problem (ISP) with the boundary data
$g(x,y)=-({\pa}/{\pa\nu(x)}-ik_+\rho(x))\varPhi_{k_+}(x,y),~x\in\G$.
For $z\in U^+_h\ba\ol{U}{}^+_H$ with $H>f_+$ and $h<f_-$ define
\be\no
&&U^s(y,z)=\int_{\G_H}\left(\frac{\pa u^s(x,y)}{\pa\nu(x)}\ol{\varPhi_{k_+}(x,z)}
-u^s(x,y)\frac{\pa\ol{\varPhi_{k_+}(x,z)}}{\pa\nu(x)}\right)ds(x)\\ \label{eq2321}
&&\qquad\qquad-\frac{i}{4\pi}\int_{\Sp^1_-}e^{ik_+\hat{x}\cdot(y'-z')}ds(\hat{x}),\quad y\in D_+
\en
which is independent of $h$. Then $U^s(y,z)$ is well-defined and solves the problem (ISP)
with the boundary data
\ben
G(y,z)=-\frac{i}{2}\left(\frac{\pa}{\pa\nu(y)}-ik_+\rho(y)\right)J_0(k_+|y-z|),\quad y\in\G.
\enn
\end{theorem}

\begin{remark} \label{im} {\rm
From Theorems \ref{thm301}, \ref{thm302}, \ref{thm309} and \ref{thm310}, and by a similar argument
as for the case of a penetrable rough surface, it can be obtained that the scattered field $U^s(y,z)$
for the case of an impenetrable rough surface has a similar behavior as for the case of a penetrable rough
surface given in Remark \ref{rmbehavior}, that is, it is expected that for any $y$ in each compact subset of $D_+$
the scattered field $U^s(y,z)$ for the case of an impenetrable rough surface takes a large value when $z\in\G$
and decays as $z$ moves away from $\G$.
}
\end{remark}

\subsection{The imaging function}

Motivated by the above discussion, we introduce the following imaging function
\be\no
&&I(z)=\int_{\G_H}\left|\int_{\G_H}\left(\frac{\pa u^s(x,y)}{\pa\nu(x)}\ol{\varPhi_{k_+}(x,z)}
-u^s(x,y)\frac{\pa \ol{\varPhi_{k_+}(x,z)}}{\pa \nu(x)}\right)ds(x)\right.\\ \label{eq18}
&&\qquad\qquad\left.-\frac{i}{4\pi}\int_{\Sp^1_-}e^{ik_+\hat{x}\cdot(y'-z')}ds(\hat{x})
\right|^2ds(y),
\en
where $u^s(x,y)$ is the scattered field to one of the three scattering problems mentioned above.
Using the asymptotic properties of $U^s(y,z)$ given in Remark \ref{rmbehavior} and \ref{im},
we can expect that the imaging function $I(z)$ takes a large value when $z\in\G$ and decays
as $z$ moves away from the rough surface $\G$.

In numerical computation, the infinite integration interval $\G_H$ in (\ref{eq18}) is truncated to be
${\G_{H,A}}:=\{x\in\Gamma_H~|~|x_1|<A\}$ which will be discretized uniformly into $2N$ subintervals so
the step size is $h=A/N.$ In addition, the lower-half circle $\Sp^1_-$ in the second integral
in (\ref{eq18}) will also be uniformly discretized into $M$ grids with the step size $\Delta\theta=\pi/M$.
Then for each sampling point $z$ we get the following discrete form of (\ref{eq18})
\be\no
&&I_A(z)=h\sum_{j=0}^{2N}\left|h\sum_{i=0}^{2N}
\left(\frac{\pa u^s(x_i,y_j)}{\pa\nu(x)}\ol{\varPhi_{k_+}(x_i,z)}
-u^s(x_i,y_j)\frac{\pa\ol{\varPhi_{k_+}(x_i,z)}}{\pa\nu(x)}\right)\right.\\ \label{eq188}
&&\qquad\qquad\qquad\qquad\left.-\frac{i\Delta\theta}{4\pi}
\sum_{k=0}^{M}e^{ik_+d_k\cdot(y_j'-z')}\right|^2.
\en
Here, the measurement points are denoted by $x_i=(-A+ih,H),$ $i=0,1,...,2N,$
the incident source positions are $y_j=(-A+jh,H),$ $j=0,1,...,2N,$
and $d_k=(\sin(-\pi+k\Delta\theta),\cos(-\pi+k\Delta\theta)),$ $k=0,1,...,M.$

The direct imaging method based on (\ref{eq188}) can be described in the following algorithm.

\begin{algorithm} Let $K$ be the sampling region which contains the part of
the rough surface that we want to recover.
\begin{enumerate}
\item Choose $\mathcal{T}_m$ to be a mesh of $K$
and $\Gamma_{H,A}$ to be a straight line segment above the rough surface.
\item Collect the Cauchy data $(u^s(x_i,y_j),{\pa u^s(x_i,y_j)}/{\pa\nu(x)})$ on the measurement points
      $x_i$, $i=0,...,2N$, on $\G_{H,A}$ corresponding to the incident point sources
      $u^i(x,y_j)=\varPhi_{k_+}(x,y_j),$ $j=0,...,2N$.
\item For each sampling point $z\in\mathcal{T}_m$, compute the approximate imaging function $I_A(z)$ in (\ref{eq188}).
\item Locate all those sampling points $z\in\mathcal{T}_m$ such that $I_A(z)$ takes a large value,
which represent the part of the rough surface on the sampling region $K$.
\end{enumerate}
\end{algorithm}

\section{Numerical examples}\label{sec4}
\setcounter{equation}{0}

In this section, we present several numerical experiments to demonstrate the effectiveness of
our imaging method and compare the reconstructed results by using different parameters.
To generate the synthetic data, we use the Nystr\"{o}m method to solve the forward scattering
problems for the case of global rough surfaces \cite{LSZR16,M01,ZC03}.
The noisy Cauchy data are generated as follows
\ben
u^s_{\delta}(x)&=&u^s(x)+\delta(\zeta_1+i\zeta_2)\max_x|u^s(x)|,\\
\pa_\nu u^s_{\delta}(x)&=&\pa_\nu u^s(x)+\delta(\zeta_1+i\zeta_2)\max_x|\pa_\nu u^s(x)|,
\enn
where $\delta$ is the noise ratio and $\zeta_1,\zeta_2$ are the standard normal distributions.
In all examples, we choose $N=100$ and $M=256$.

In each figure, we use a solid line to represent the actual
rough surface against the reconstructed rough surface.

{\bf Example 1.} In this example, $\G_1$ and $\G_2$ are two Dirichlet rough surfaces with
\ben
\Gamma_1:\;\;f_1(x_1)&=& 0.8+0.1\sin(2\pi x_1)+0.1\sin(\pi x_1),\\
\Gamma_2:\;\;f_2(x_1)&=& 0.8+0.025\sin(5\pi(x_1-1))+0.1\sin(0.5\pi(x_1-1)).
\enn
The Cauchy data are measured on $\Gamma_{H,A}$ with $H=1.5,~A=10$.
Fig. \ref{fig1} presents the reconstructed surfaces from noisy data with $20\%$ noise for
the wave numbers $k_+=10,20,30$, respectively.
From Fig. \ref{fig1} it can be seen that the macro-scale features of the rough surface are captured
with a smaller wave number $k_+=10$ and the whole rough surface is accurately recovered with a larger wave
number $k_+=30$.

\begin{figure}[htbp]
  \centering
  \subfigure[\textbf{$k_+=10$}]{\includegraphics[width=1.65in]{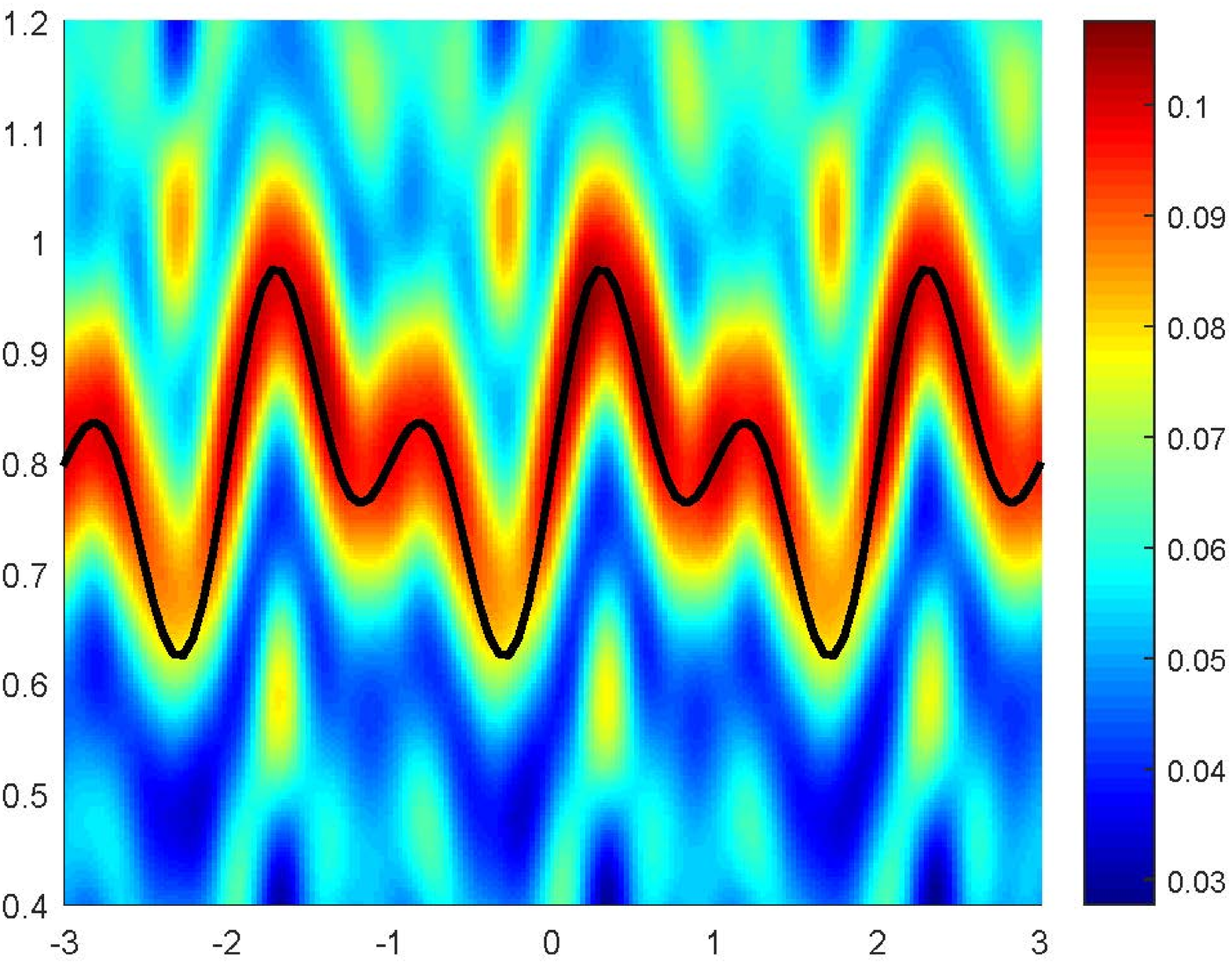}}
  \subfigure[\textbf{$k_+=20$}]{\includegraphics[width=1.65in]{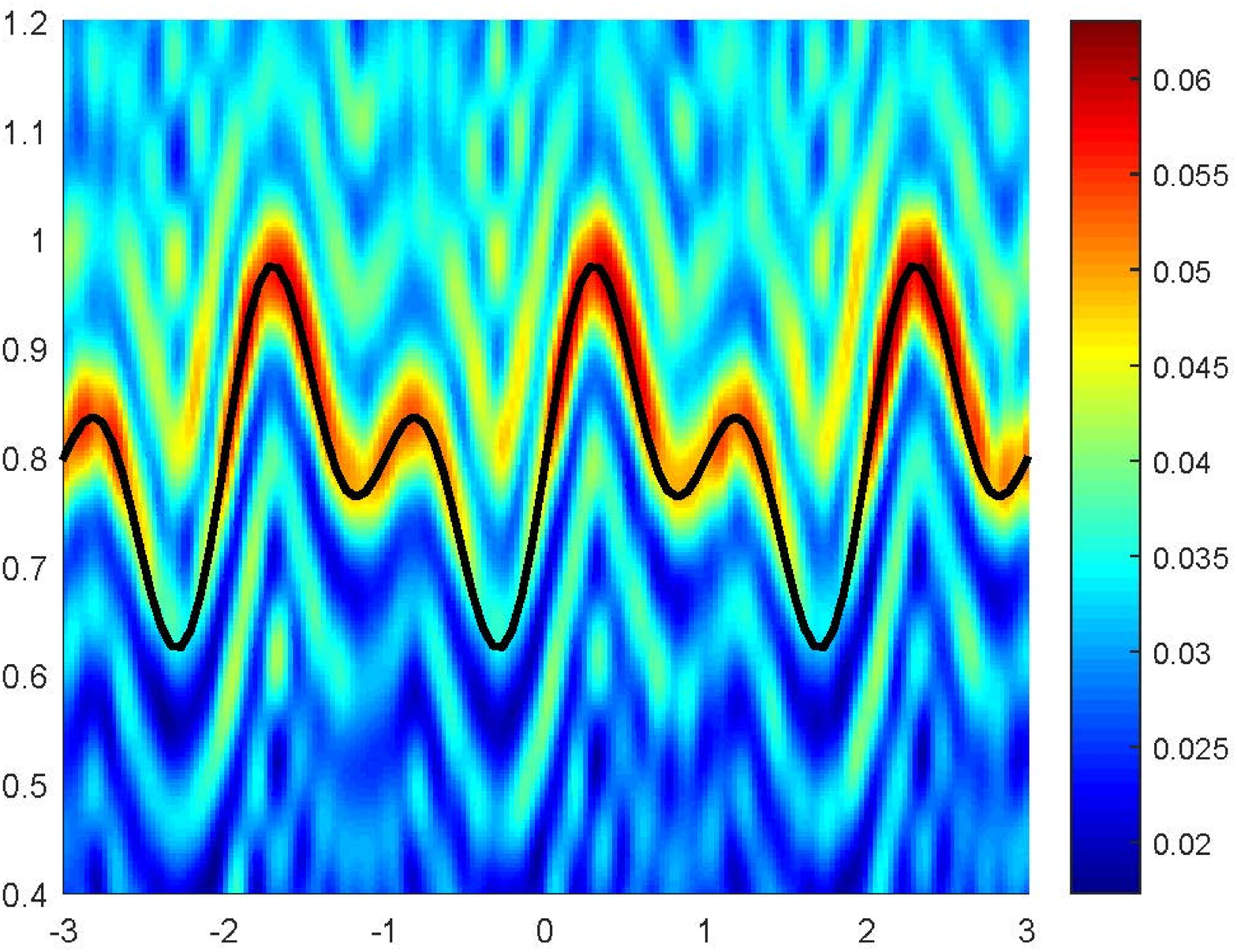}}
  \subfigure[\textbf{$k_+=30$}]{\includegraphics[width=1.65in]{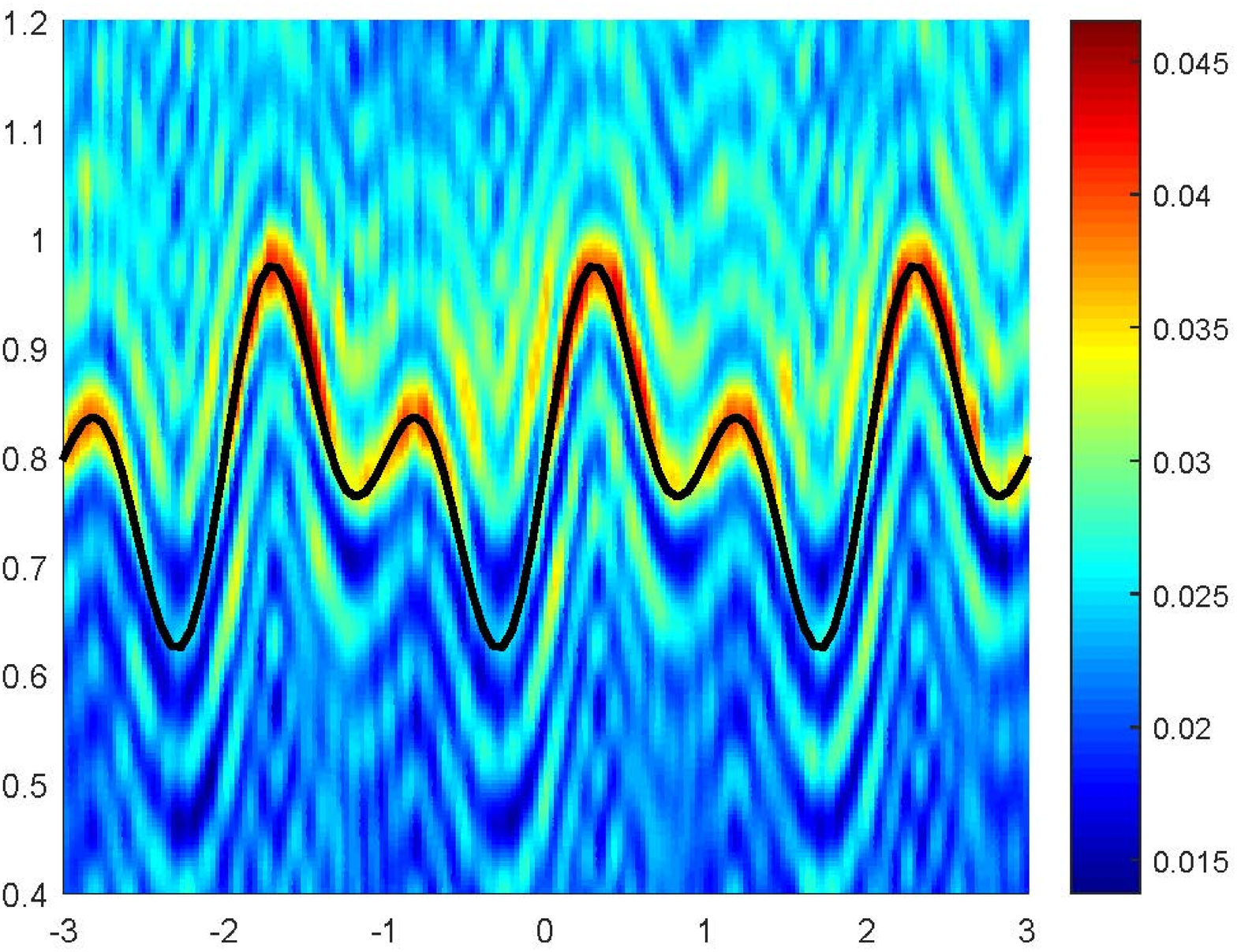}}
  \subfigure[\textbf{$k_+=10$}]{\includegraphics[width=1.65in]{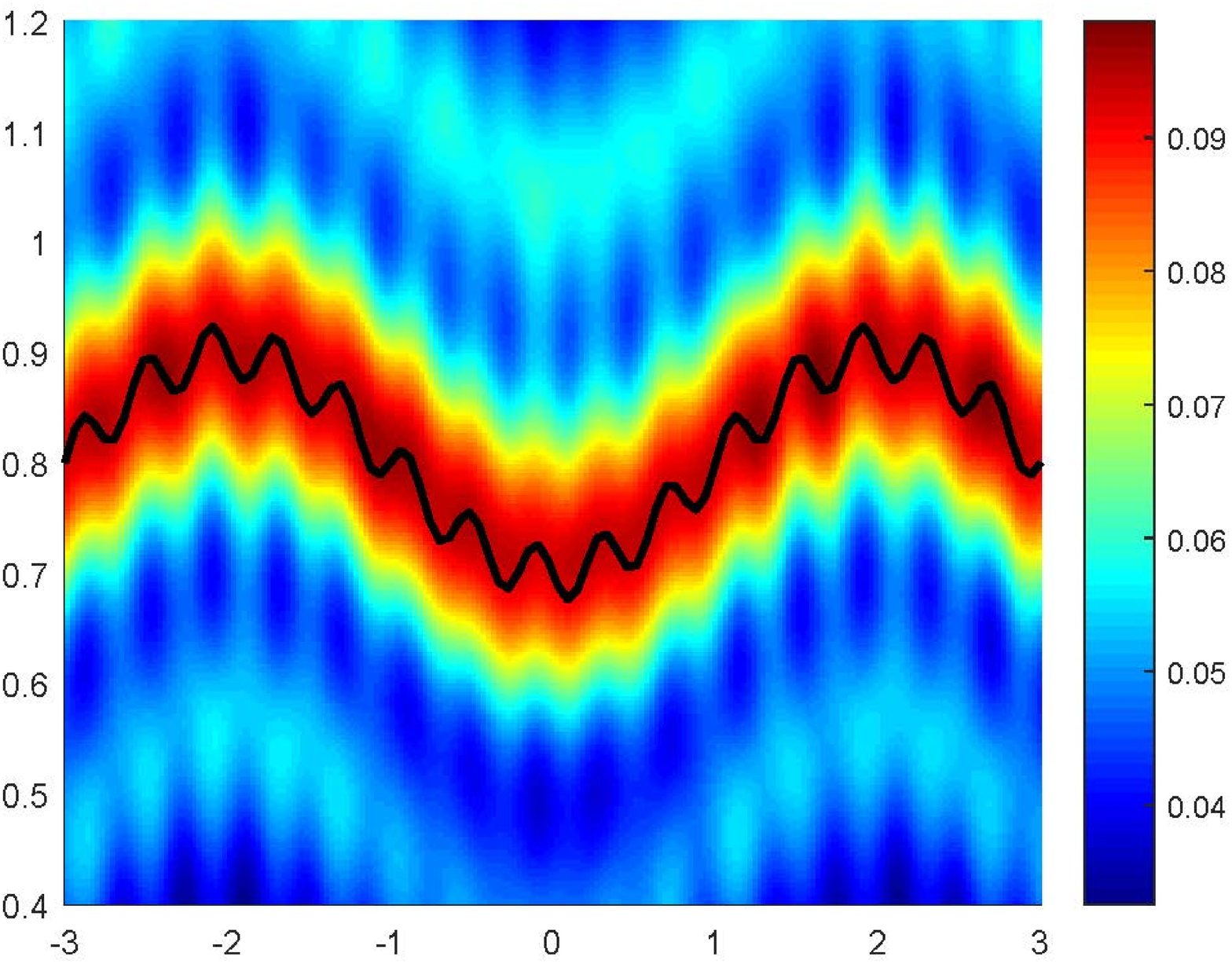}}
  \subfigure[\textbf{$k_+=20$}]{\includegraphics[width=1.65in]{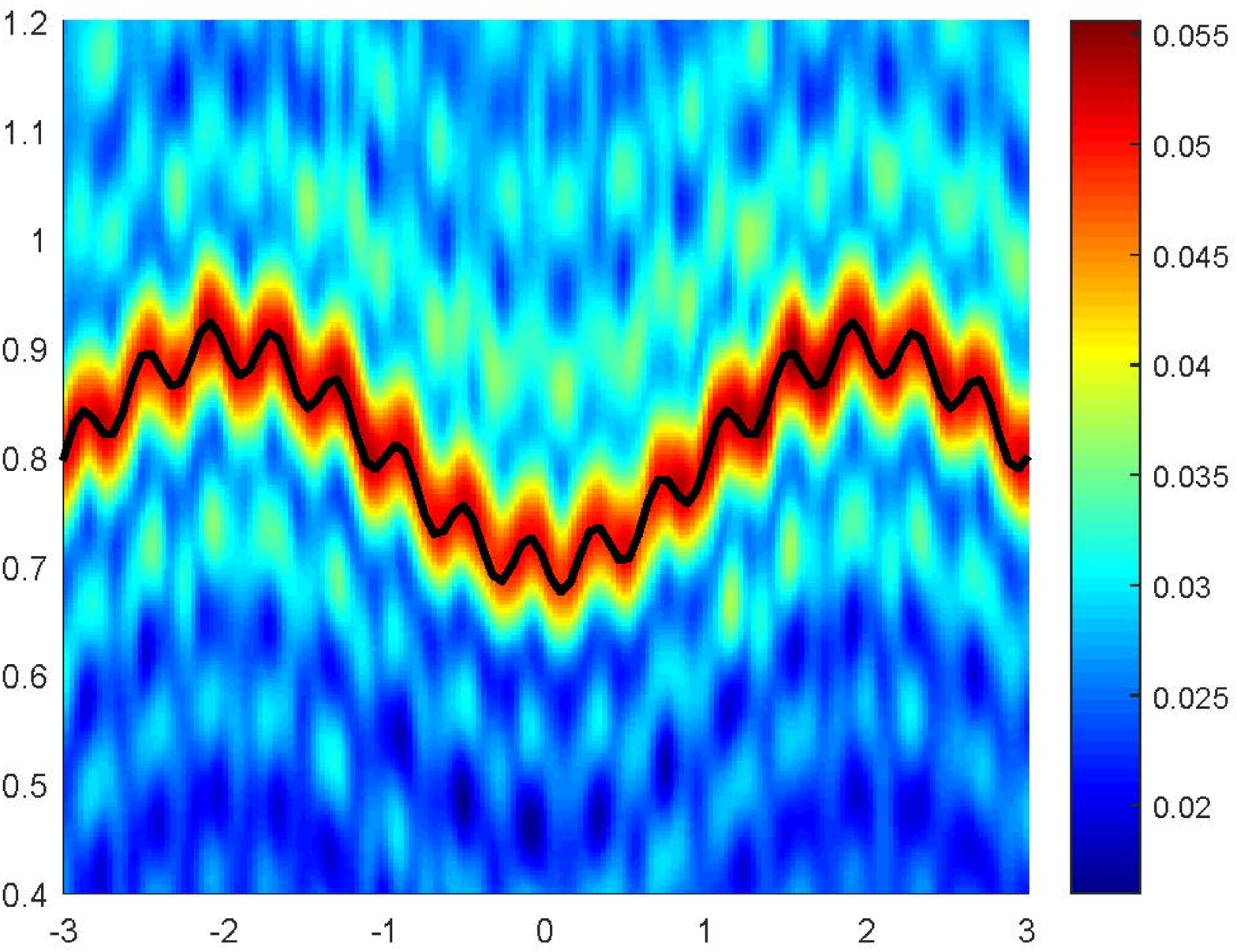}}
  \subfigure[\textbf{$k_+=30$}]{\includegraphics[width=1.65in]{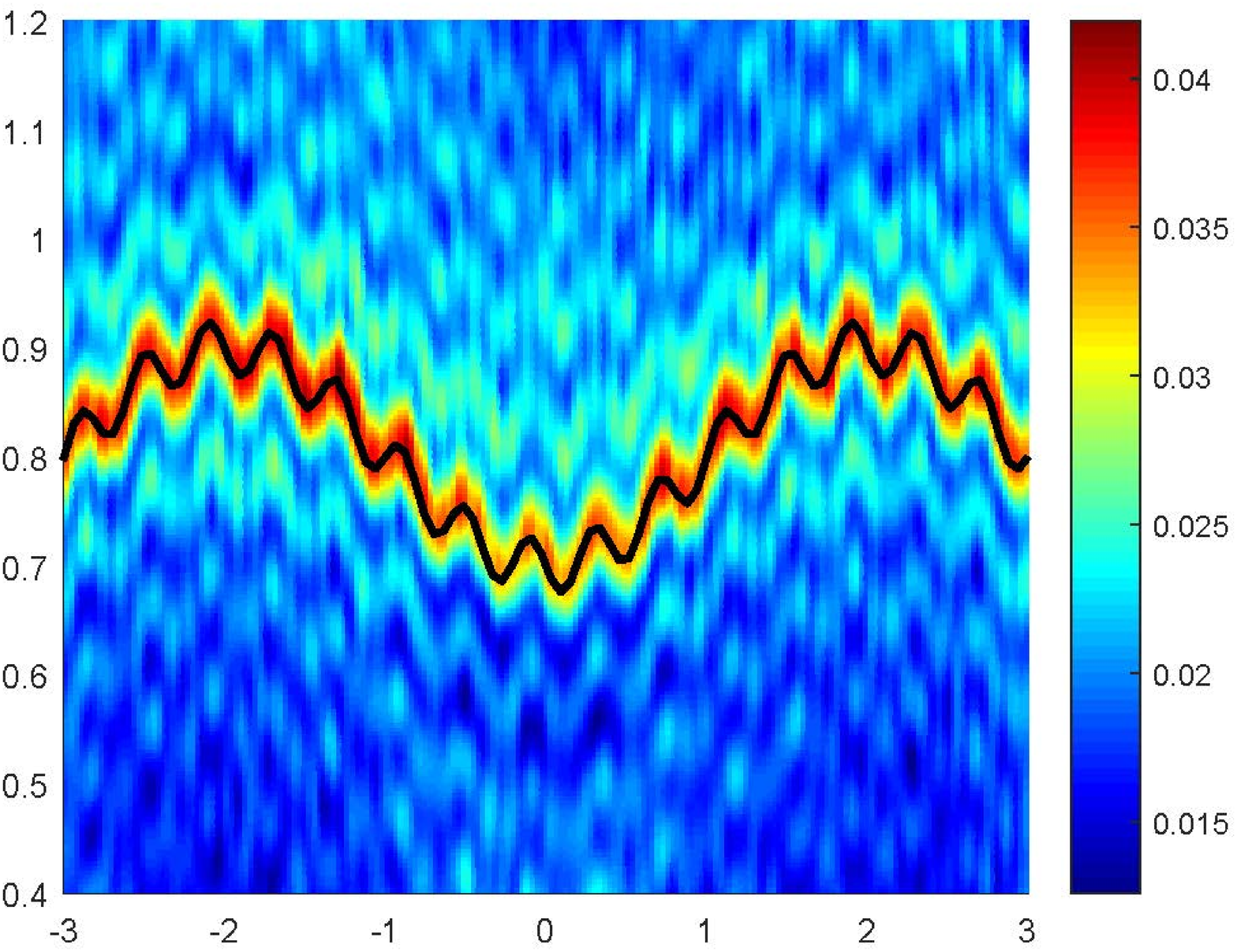}}
\caption{Reconstruction of Dirichlet rough surfaces at different wave numbers.
The top row is the reconstructed result of $\G_1$ and the bottom row is the reconstructed result of $\G_2$.
}\label{fig1}
\end{figure}

{\bf Example 2.} We now consider two impedance rough surfaces $\Gamma_3$ and $\Gamma_4$ with
\ben
\Gamma_3:\;\;f_3(x_1)&=& 0.8+0.16\sin(\pi x_1), \\
\Gamma_4:\;\;f_4(x_1)&=& 0.8+ 0.1e^{-25(0.3x_1-0.5)^2}+0.2e^{-49(0.3x_1+0.6)^2}-0.25e^{-8x_1^2}.
\enn
The wave number and the impedance function are set to be $k_+=15$ and $\rho(x_1,f(x_1))=5+\exp({2\pi x_1 i})$,
respectively. Fig. \ref{fig2} presents the reconstructed surfaces from noisy data with $20\%$ noise for
the cases when the Cauchy data are measured on $\Gamma_{H,A}$ with $H=1.5,~A=4$, $H=1.5,~A=10$ and $H=3,~A=10$,
respectively. The reconstruction results show that the reconstruction is getting better
if the measurement surface $\Gamma_{H,A}$ is getting closer to the rough surface
and is also getting longer.

\begin{figure}[htbp]
  \centering
  \subfigure[\textbf{$\{(x_1,1.5)|\;|x_1|\leq 4\}$}]{\includegraphics[width=1.65in]{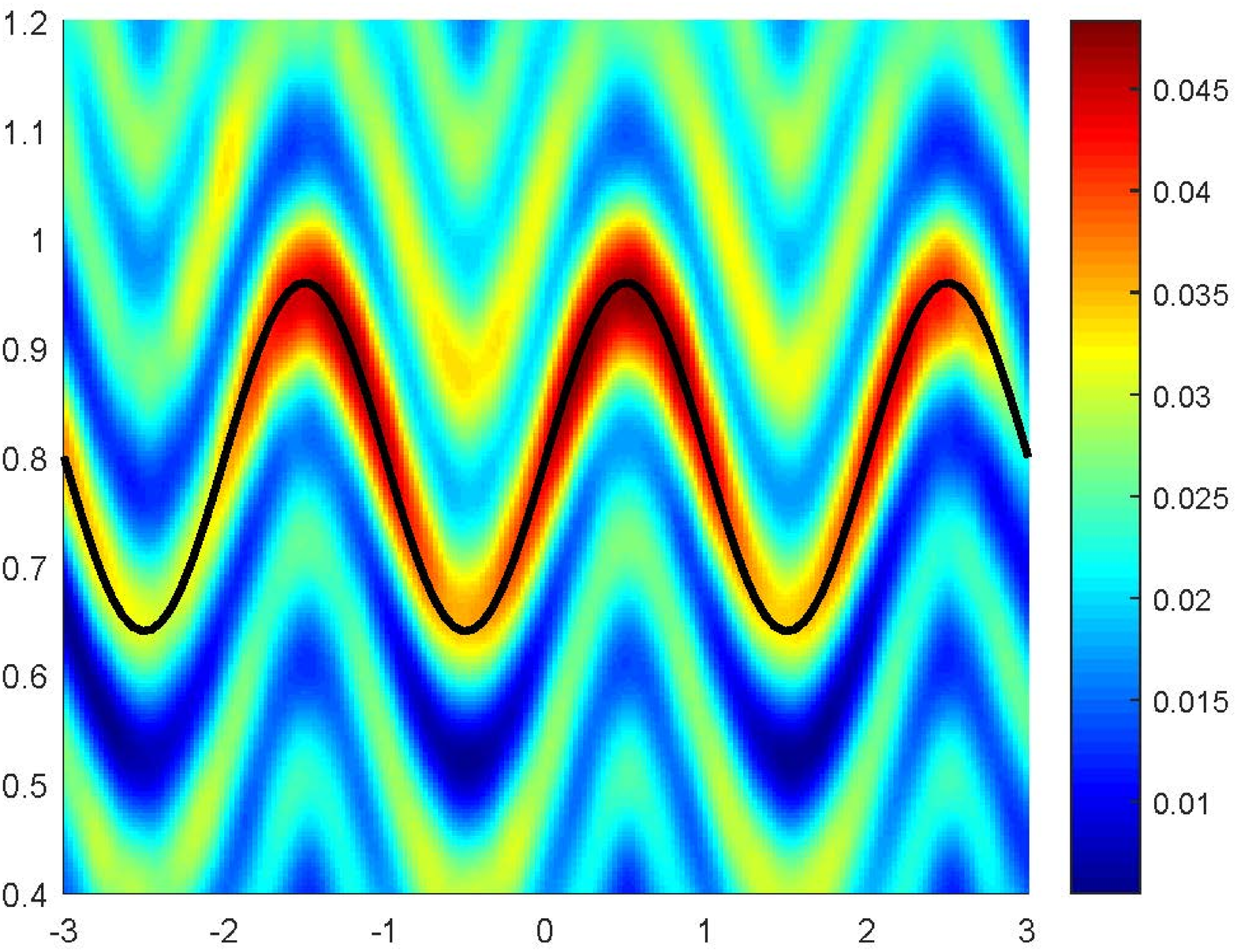}}
  \subfigure[\textbf{$\{(x_1,1.5)|\;|x_1|\leq 10\}$}]{\includegraphics[width=1.65in]{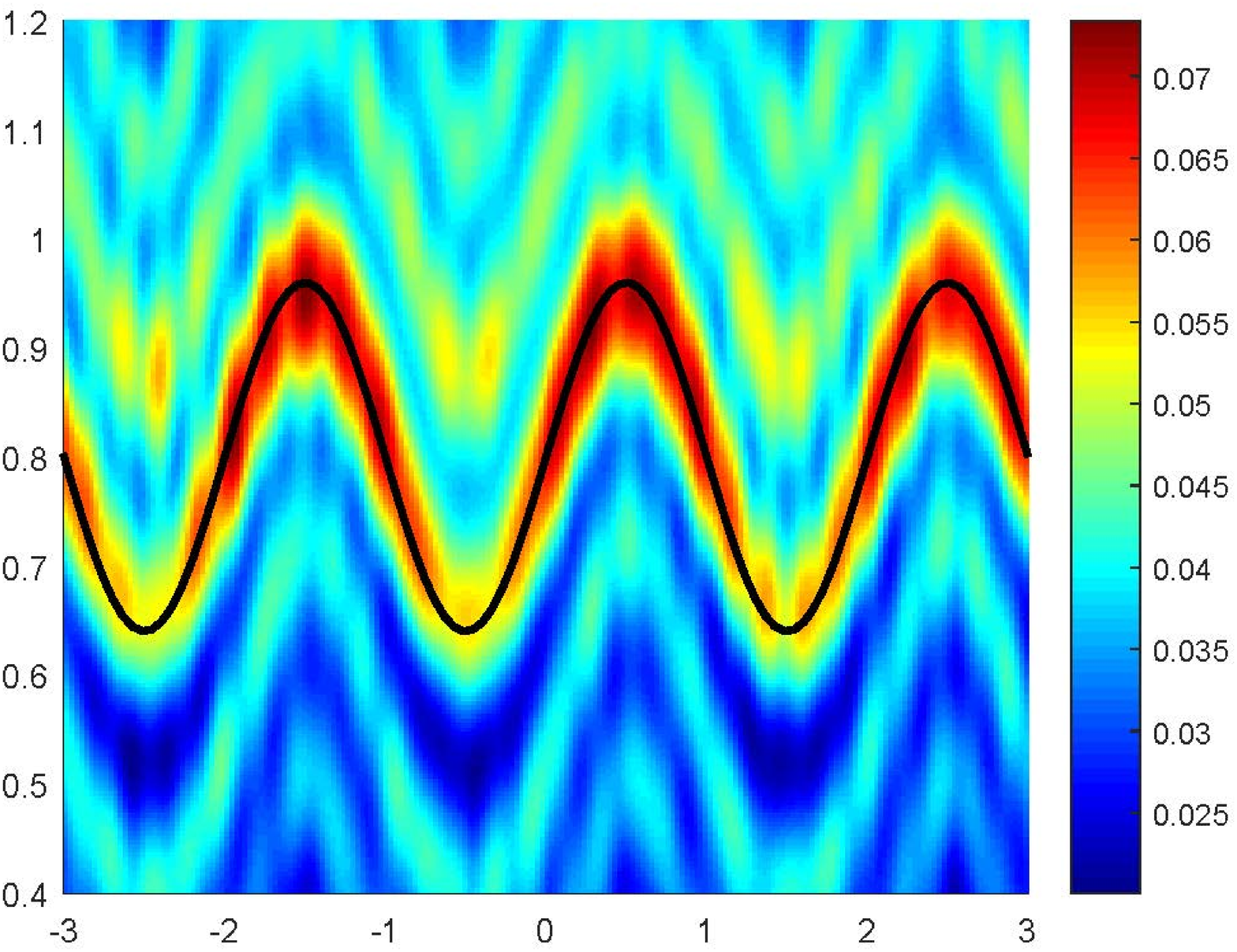}}
  \subfigure[\textbf{$\{(x_1,3)|\;|x_1|\leq 10\}$}]{\includegraphics[width=1.65in]{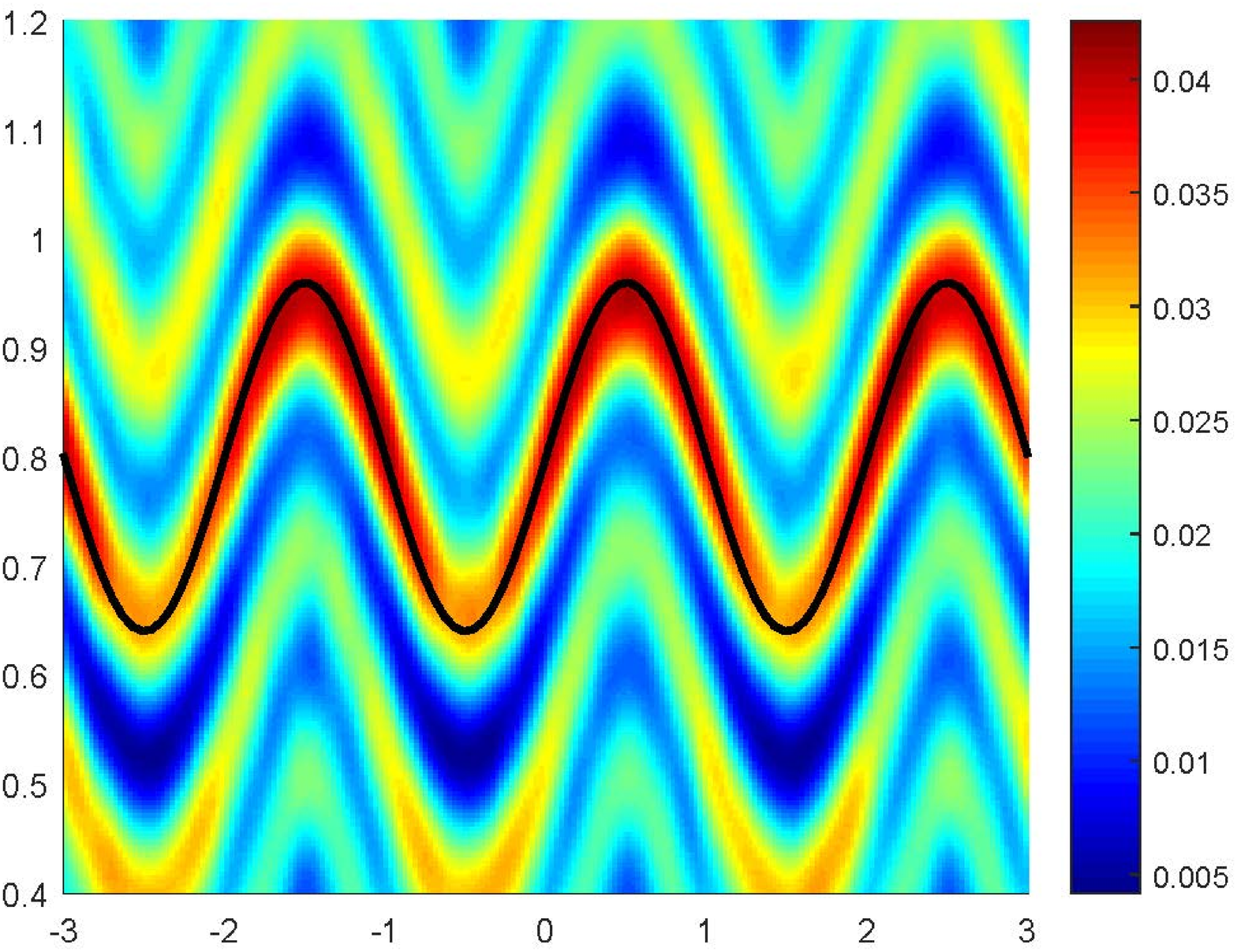}}
  \subfigure[\textbf{$\{(x_1,1.5)|\;|x_1|\leq 4\}$}]{\includegraphics[width=1.65in]{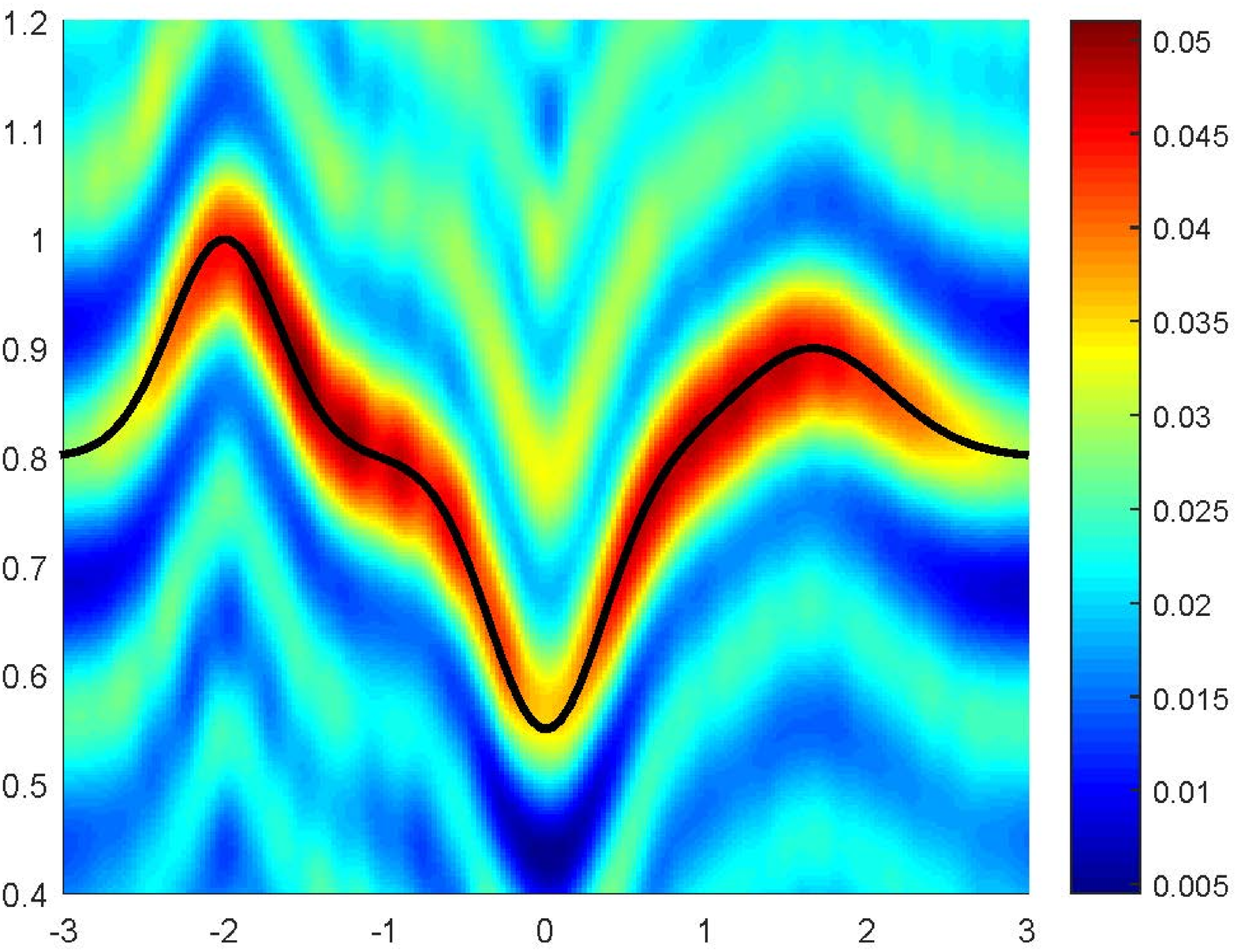}}
  \subfigure[\textbf{$\{(x_1,1.5)|\;|x_1|\leq 10\}$}]{\includegraphics[width=1.65in]{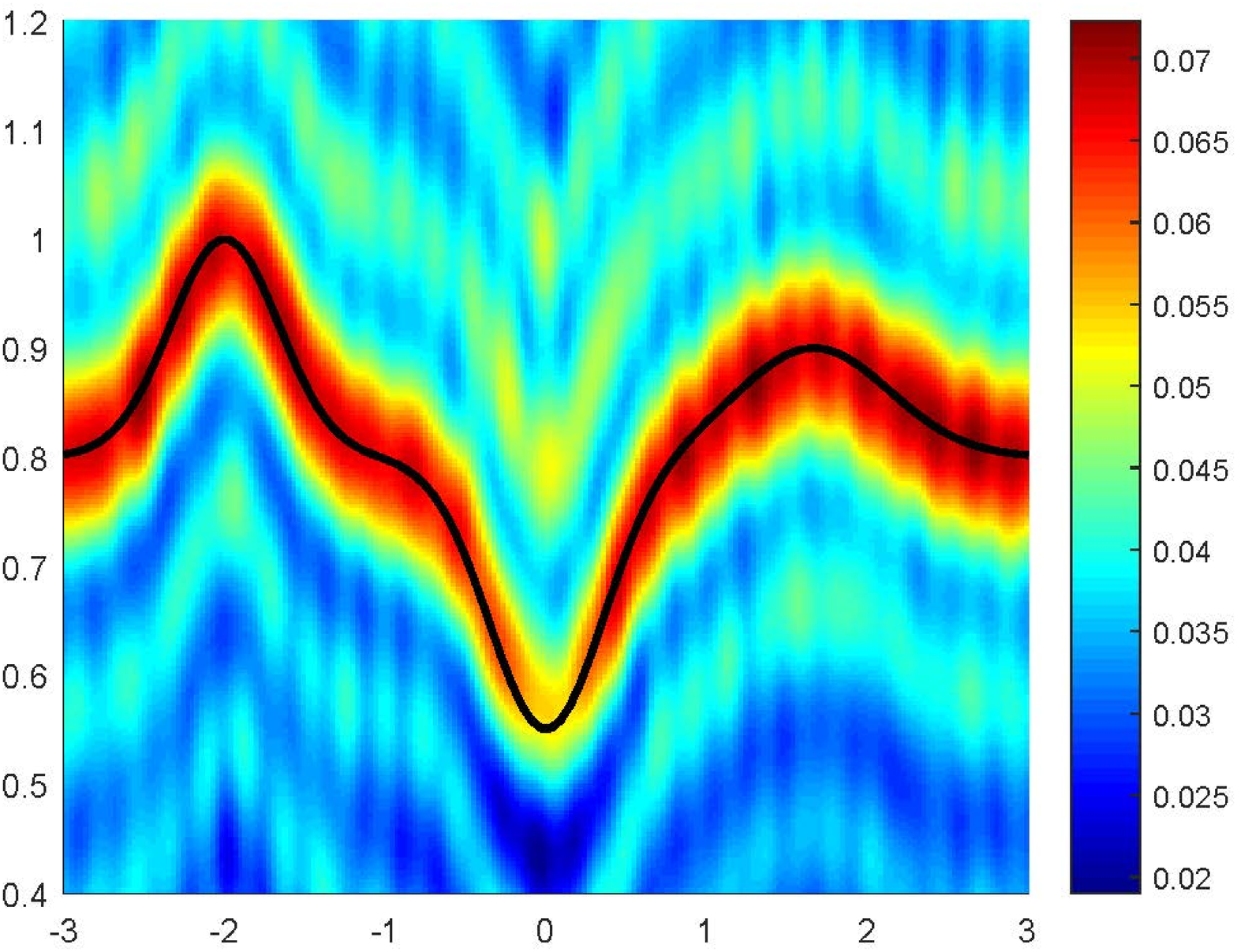}}
  \subfigure[\textbf{$\{(x_1,3)|\;|x_1|\leq 10\}$}]{\includegraphics[width=1.65in]{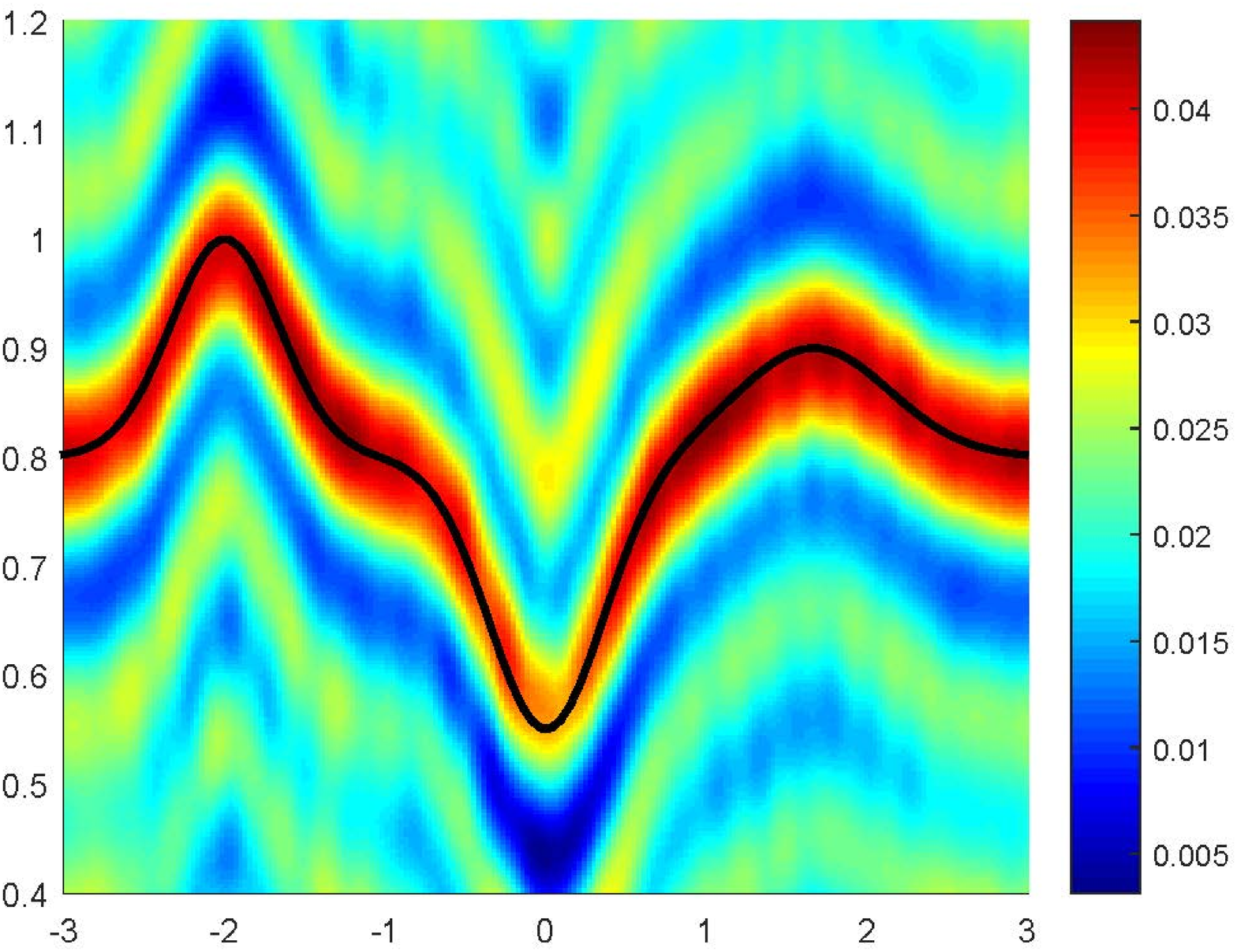}}
\caption{Reconstruction of an impedance rough surface with different measurement places and different
lengths of measurement line segments.
The top row is the reconstruction results of $\Gamma_3$
and the bottom row is the reconstructed results of $\Gamma_4$.
}\label{fig2}
\end{figure}

{\bf Example 3.} This example considers two penetrable rough surfaces $\Gamma_5$ and $\Gamma_6$:
\ben
\Gamma_5:\;\;f_5(x_1)&=& 0.8+0.3\sin(0.7\pi x_1)e^{-0.4x_1^2},\\
\Gamma_6:\;\;f_6(x_1)&=& 0.8+0.1\sin(0.4\pi x_1)e^{-\sin(1.2x_1^2)}.
\enn
The wave numbers are set to be $k_+=20$, $k_-=8$, and the Cauchy data are measured on $\G_{H,A}$ with
$H=1.5,~A=10$. Fig. \ref{fig3} presents the reconstructed results from data without noise,
with $20\%$ noise and $40\%$ noise, respectively.
\begin{figure}[htbp]
  \centering
  \subfigure[\textbf{No noise }]{\includegraphics[width=1.65in]{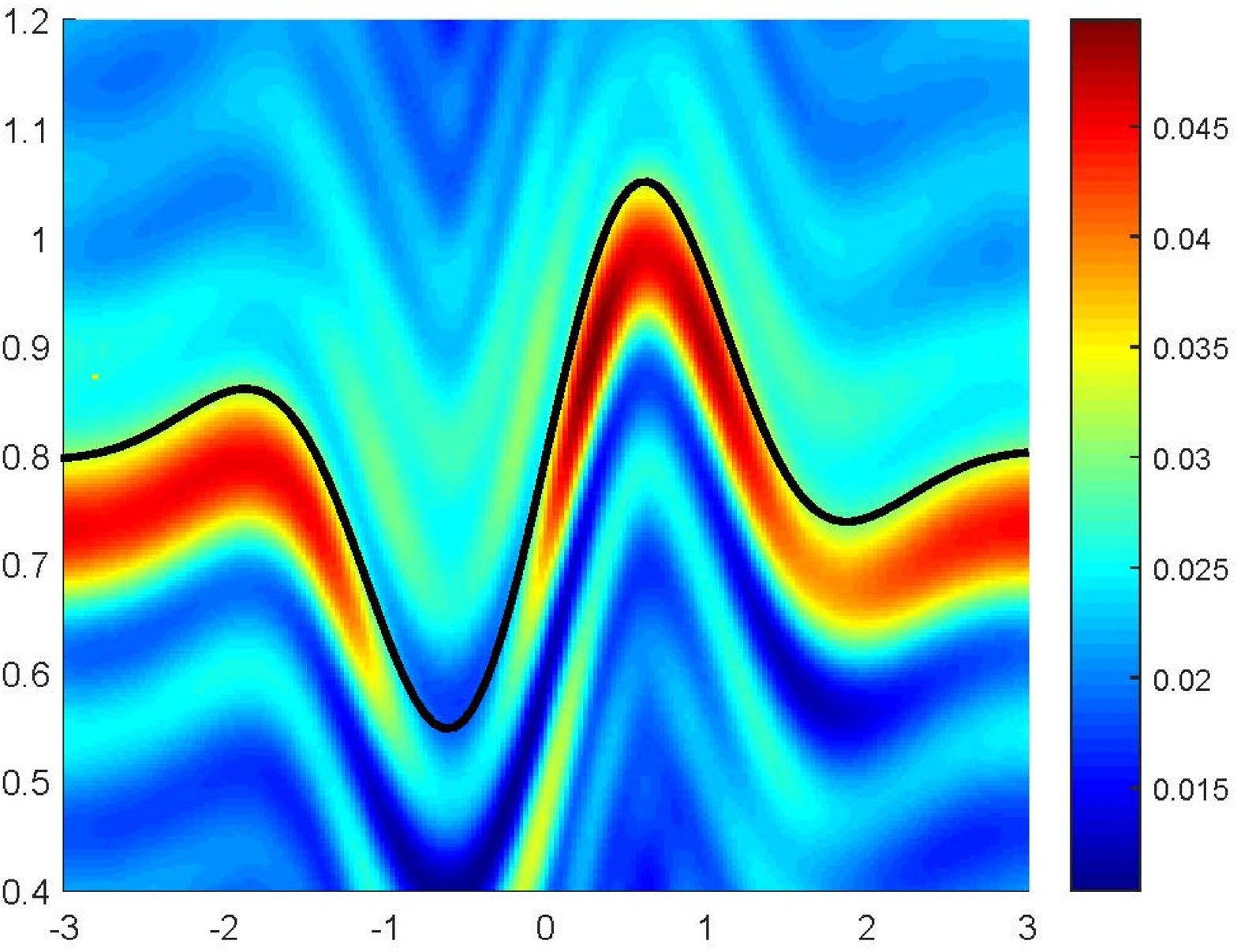}}
  \subfigure[\textbf{20\% noise }]{\includegraphics[width=1.65in]{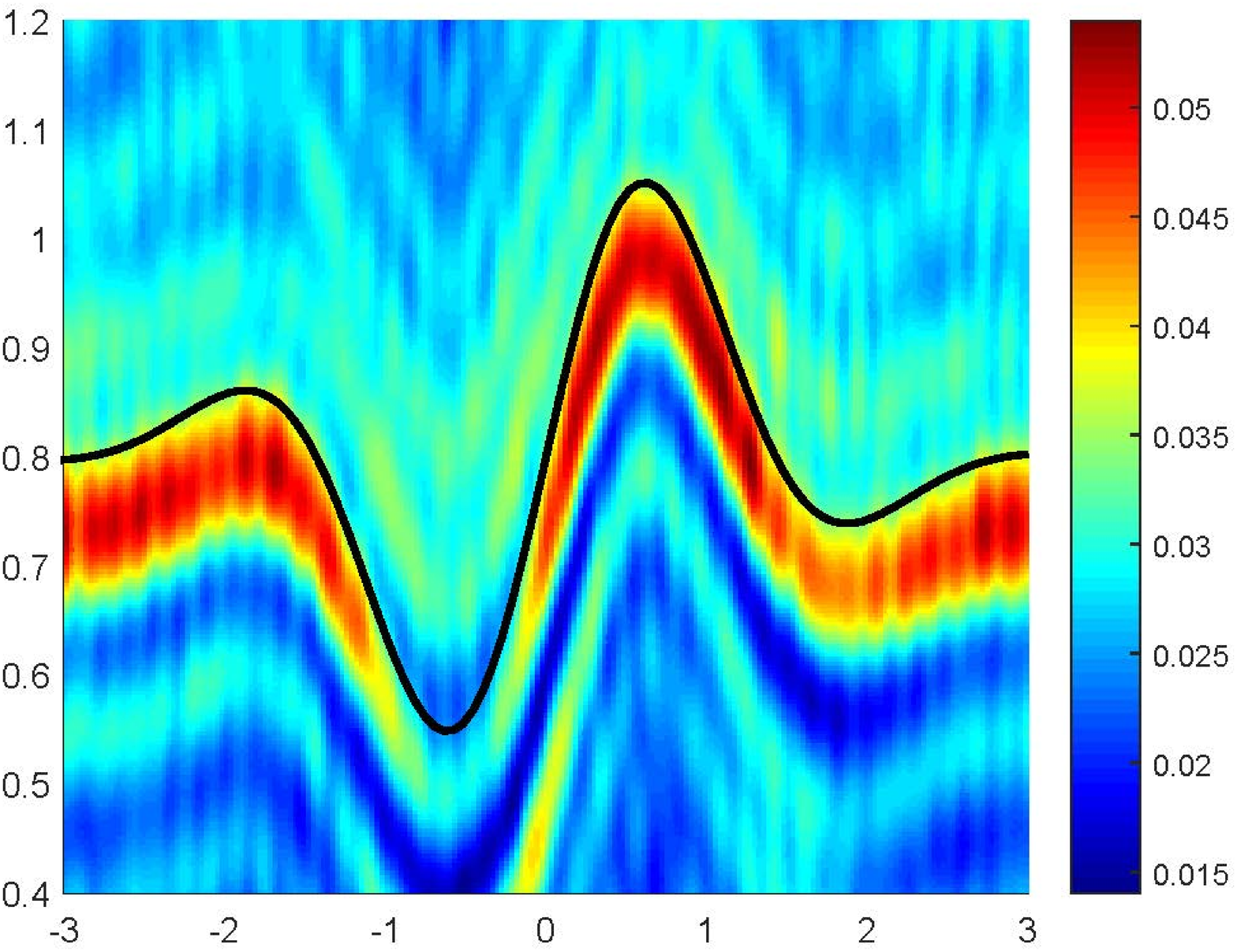}}
  \subfigure[\textbf{40\% noise }]{\includegraphics[width=1.65in]{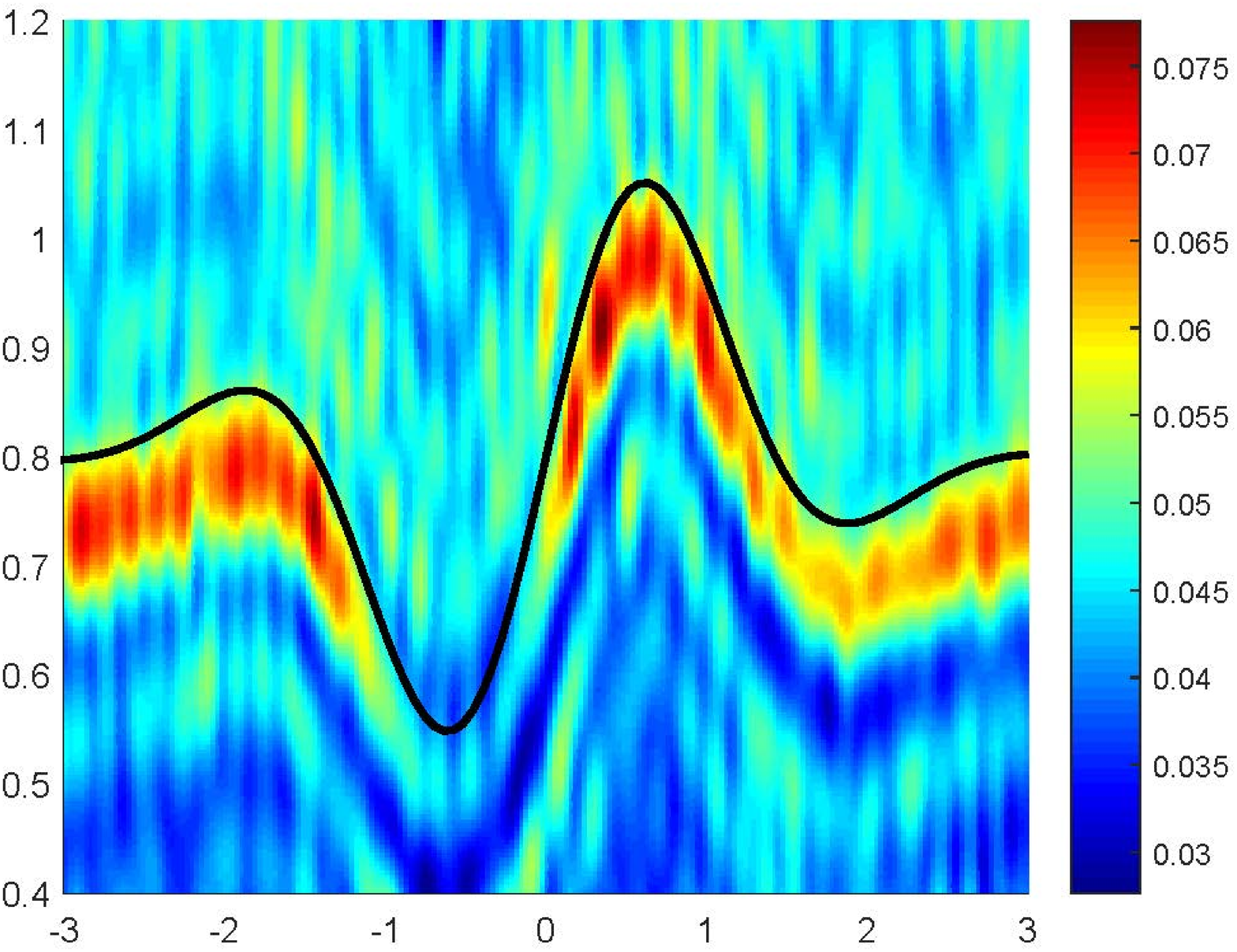}}
  \subfigure[\textbf{No noise }]{\includegraphics[width=1.65in]{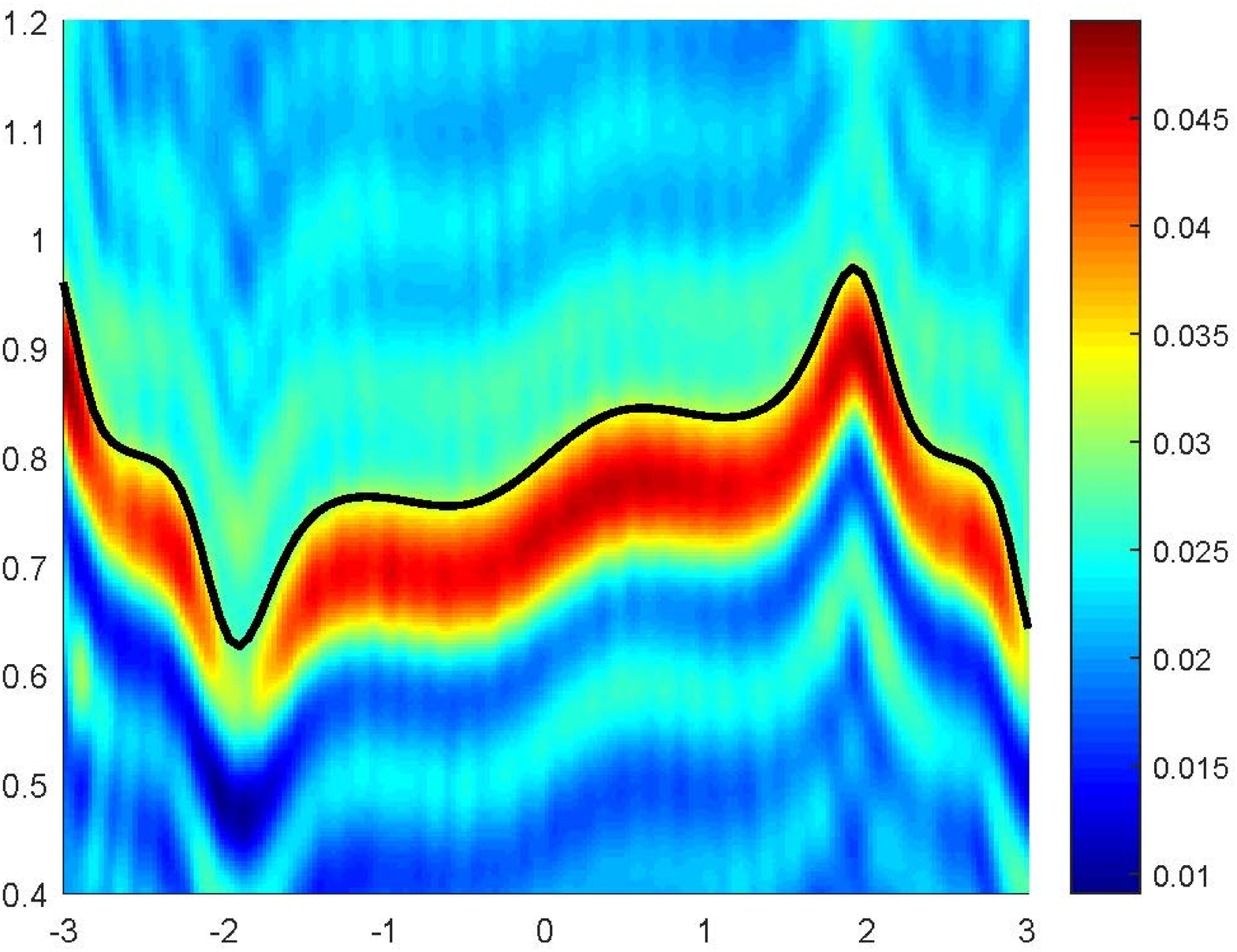}}
  \subfigure[\textbf{20\% noise }]{\includegraphics[width=1.65in]{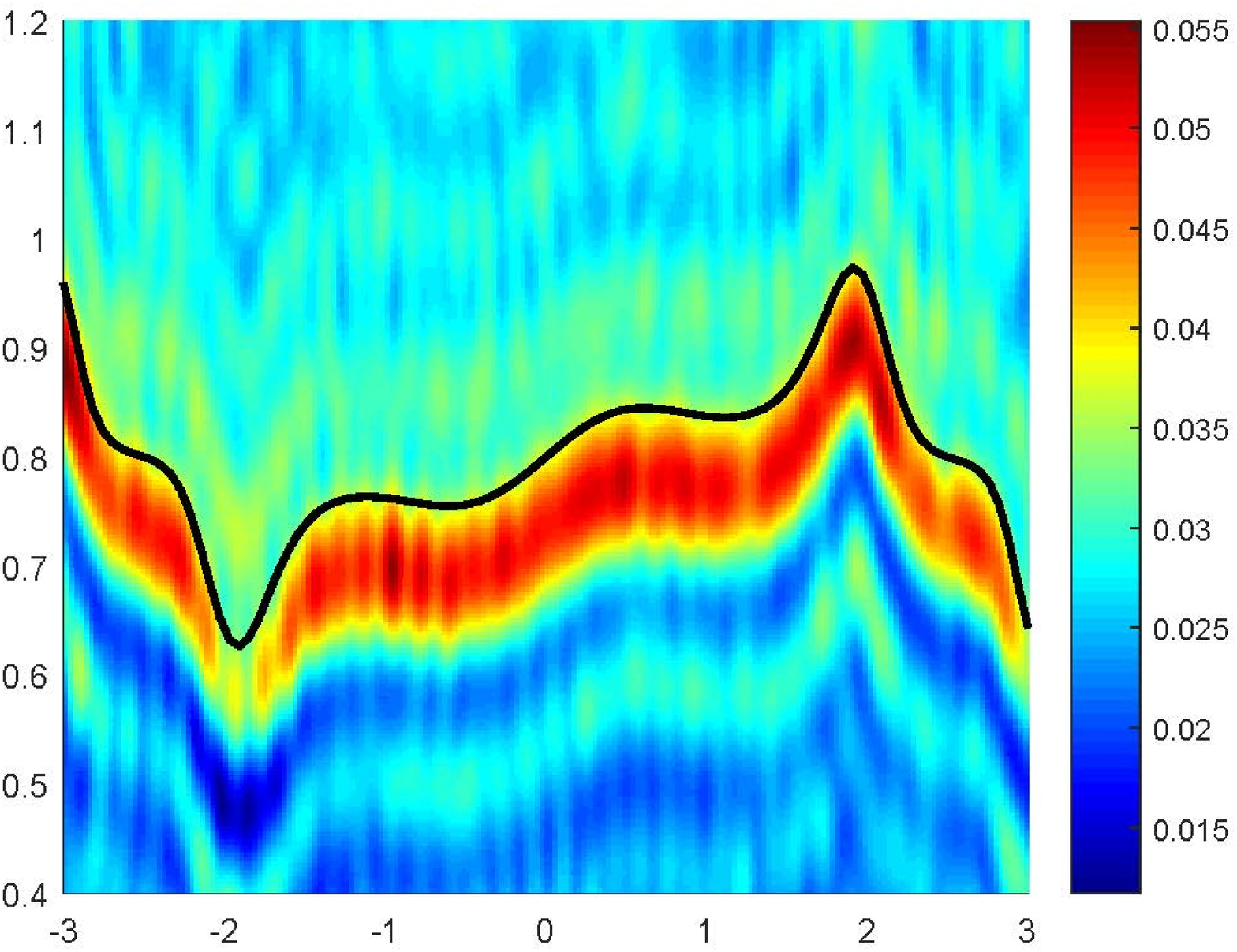}}
  \subfigure[\textbf{40\% noise }]{\includegraphics[width=1.65in]{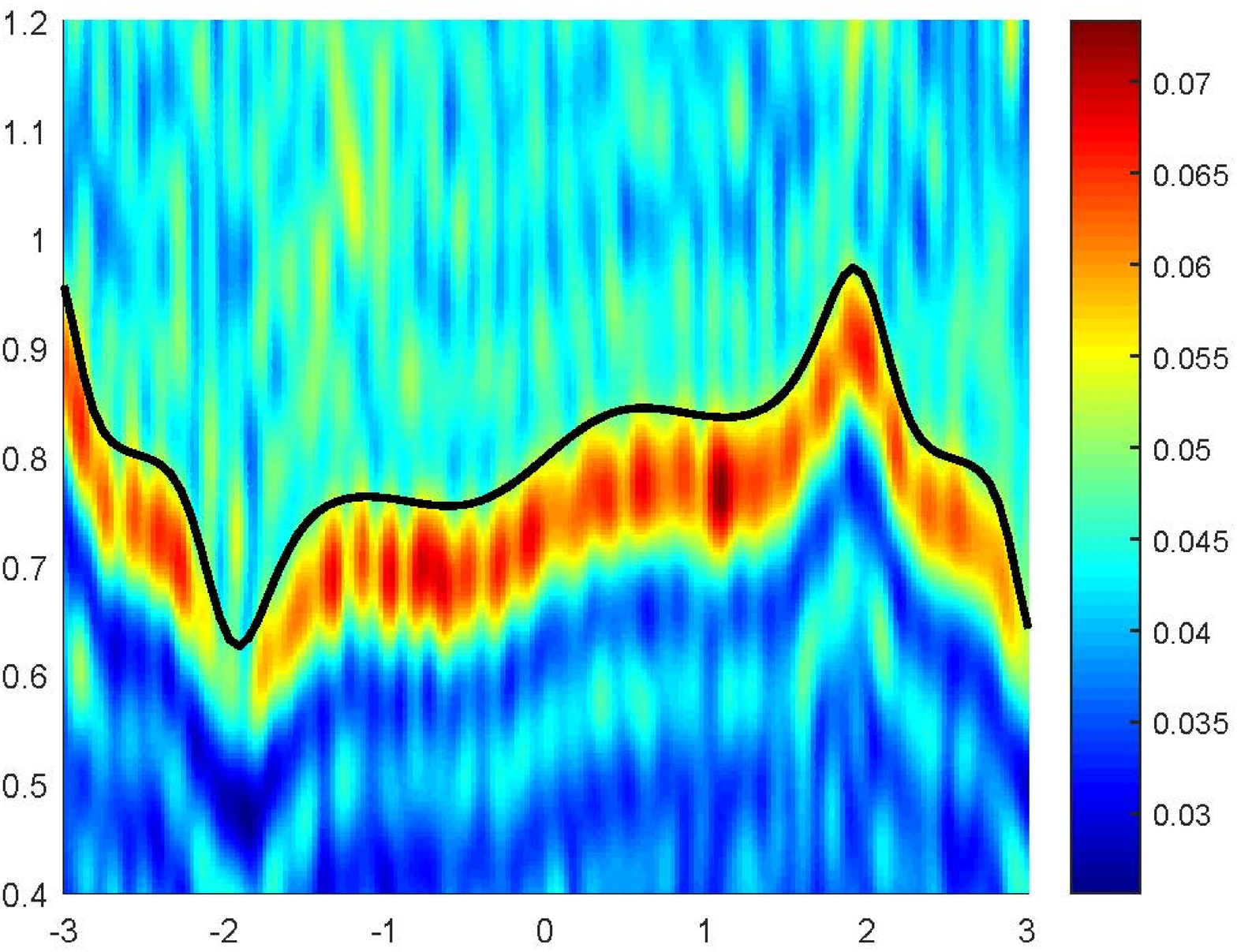}}
\caption{Reconstruction of a penetrable rough surface from data at different noise levels.
The top row is the reconstructions of $\Gamma_5$ and
the bottom row is the reconstructed results of $\Gamma_6$.
}\label{fig3}
\end{figure}

The above numerical examples and several other examples carried out but not presented here
illustrate that the direct imaging method gives an accurate and stable reconstruction of
unbounded rough surfaces. The method is very robust to noise in the measured data and is
independent of the physical property of the rough surfaces.

\section{Conclusion}

We proposed a direct imaging method for inverse scattering problems by an unbounded rough surface.
Our imaging method does not need to know the property of the rough surface in advance, so it can be used
to reconstruct both penetrable and impenetrable rough surfaces.
Numerical experiments have also been carried out to show that the reconstruction is accurate and robust to noise.
Further, our imaging method can be extended to many other cases such as inverse elastic scattering problems by
unbounded rough surfaces. We will report such results in a forthcoming paper.

\section*{Acknowledgements}

This work is partly supported by the NNSF of China grants 91630309, 11501558 and 11571355
and the National Center for Mathematics and Interdisciplinary Sciences, CAS.

\end{document}